\documentclass[a4paper]{amsart}
\usepackage{graphicx}
\usepackage{amssymb}
\usepackage{amsmath}
\usepackage{amsthm}
\usepackage{amscd}
\usepackage[all,2cell]{xy}

\UseAllTwocells \SilentMatrices
\newtheorem{thm}{Theorem}[section]

\newtheorem{cor}[thm]{Corollary}
\newtheorem{lem}[thm]{Lemma}
\newtheorem{exm}[thm]{Example}

\newtheorem{prop}[thm]{Proposition}
\theoremstyle{definition}
\newtheorem{defn}[thm]{Definition}
\theoremstyle{remark}
\newtheorem{rem}[thm]{\bf Remark}
\numberwithin{equation}{section}

\begin{document}
\title[Monadicity theorem and weighted projective lines of tubular type]
{Monadicity theorem and weighted projective lines of tubular type}

\author[Jianmin Chen, Xiao-Wu Chen, Zhenqiang Zhou] {Jianmin Chen, Xiao-Wu Chen$^*$, Zhenqiang Zhou}

\thanks{$^*$ the corresponding author}
\subjclass[2010]{16W50, 14A22, 14H52}
\date{\today}
\keywords{weighted projective line, elliptic curve, equivariantization, monad, graded ring}%
\maketitle

\dedicatory{}%
\commby{}%

\begin{abstract}
We formulate a version of Beck's monadicity theorem for abelian categories, which is applied to the  equivariantization of abelian categories with respect to a finite group action.  We prove that the equivariantization is compatible with the construction of quotient abelian categories by  Serre subcategories. We prove that the equivariantization of the graded module category over a graded ring is equivalent to the  graded module category over the same ring but with a different grading.  We deduce from these results two equivalences between the category of (equivariant) coherent sheaves on a weighted projective line of tubular type and that on an elliptic curve, where the acting groups are cyclic and the two equivalences are somehow adjoint to each other.
\end{abstract}

\section{Introduction}

The close relationship between weighted projective lines of tubular type and elliptic curves is known to experts since the creation of weighted projective lines in \cite{GL87}. Roughly speaking, the category of coherent sheaves on a weighted projective line of tubular type is equivalent to the category of equivariant coherent sheaves on an elliptic curve with respect to a finite abelian group action. This result is  implicitly  exploited in \cite{Po} using ramified elliptic Galois coverings over the field of complex numbers; compare \cite{Ploog}.  Moreover, the somehow converse result holds true by \cite{Len, Lentalk}: the category of coherent sheaves on an elliptic curve is equivalent to the category of equivariant coherent sheaves on a weighted projective line of tubular type with respect to a different finite abelian group action, which contains the action given by the Auslander-Reiten translation.

The goal of this paper is to re-exploit the above mentioned equivalences in an explicit form. We emphasize that our treatment of the four tubular types is in a uniform  manner, and that we deduce these equivalences from general results on the equivariantization of abelian categories. These general results are direct applications of the famous Beck's monadicity theorem which characterizes the module category over a monad; see \cite[Chapter VI]{McL}. We mention that monads appear naturally, since the category of equivariant objects is isomorphic to the module category over a certain monad; see \cite{Ch,El2014}.

Let us describe the content of this paper.  In Section 2, we collect some basic facts on monads and modules over monads. We formulate a version of Beck's monadicity theorem for abelian categories, which is convenient for applications; see Theorem \ref{thm:Beck}. We give a self-contained proof to Theorem \ref{thm:Beck}. In Section 3, we recall the notions of a group action on a category and the category of equivariant objects. The observation we need is that the category of equivariant objects is isomorphic to the category of modules over a certain monad; see Proposition \ref{prop:monadic}. In Section 4, we apply Theorem \ref{thm:Beck} to prove that the equivariantization of an abelian category is compatible with quotient abelian categories by Serre subcategories; indeed, a slightly more general result on the category of modules over an exact monad is obtained; see Proposition \ref{prop:app1} and Corollary \ref{cor:equiv}. In Section 5, we  prove that the equivarianzation of the graded module category over a graded ring with respect to a degree-shift action is equivalent to the graded module category over the same ring but with a coarser grading; see Proposition \ref{prop:app2}. This enables us to recover \cite[Theorem 6.4.1]{NV} on the graded module category over the corresponding graded group ring; see Corollary \ref{cor:NV}. For the case that the acting group is finite abelian, we generalize slightly a result in \cite{Len} and prove that the equivariantization of the graded module category over a graded ring with respect to a twisting action is equivalent to the graded module category over the same ring but with a refined grading; see Proposition \ref{prop:recover}.

In Section 6, we recall some basic facts on the homogeneous coordinate algebra of a weighted projective line. We prove that the quotient graded module  category over a restriction subalgebra of the homogeneous coordinate algebra is equivalent to a quotient  of the category of coherent sheaves on the weighted projective line by  a certain Serre subcategory; see Proposition \ref{prop:eff}. This motivates the notion of an effective subgroup of the grading group for the homogeneous coordinate algebra; see Definition \ref{defn:1}. For these subgroups, Proposition \ref{prop:eff} claims an equivalence between the quotient graded module category over the restriction subalgebra and the category of coherent sheaves on the weighted projective line.

The final section is devoted to the study of the relationship between weighted projective lines of tubular type and elliptic plane curves. We prove that the homogeneous coordinate algebra of an elliptic plane curve is isomorphic to the restriction subalgebra of the homogeneous coordinate algebra of the weighted projective line with respect to a suitable effective subgroup; see Theorem \ref{thm:R-S}. Applying Theorem \ref{thm:R-S} and the results in Sections 4 and 5,  we  prove the main result: the equivariantization of the category of coherent sheaves on a weighted projective line of tubular type with respect to a certain degree-shift action is equivalent to the category of coherent sheaves on an elliptic plane curve; the equivariantization of the category of coherent sheaves on an elliptic plane curve with respect to a certain twisting action is equivalent to the category of coherent sheaves on a weighted projective line of tubular type; see Theorem \ref{thm:main}. We mention that here the acting groups are cyclic, and that we need two different kinds of group actions to have these two equivalences. On the other hand, the two equivalences obtained are somehow adjoint to each other; see Remark \ref{rem:main}(1).

We emphasize that the two equivalences in Theorem \ref{thm:main} are contained in \cite{GL87} and \cite{Len, Lentalk}, which are presented in a slightly different form with sketched proofs;  compare \cite{Hille} and Remark \ref{rem:main}(3). As pointed out in \cite{GL87}, these equivalences link the classification of indecomposable coherent sheaves on elliptic curves in \cite{At} with the classification of indecomposable modules of tubular algebras in \cite{Rin}.

\section{Modules over monads}

In this section, we recall some basic facts from \cite[Chapter VI]{McL}  on monads and modules over monads. We formulate a version of Beck's monadicity theorem for abelian categories, for which we provide a self-contained proof.

Let $\mathcal{C}$ be a category. Recall that a \emph{monad} on  $\mathcal{C}$ is a triple $(M, \eta, \mu)$ consisting of an endofunctor $M\colon \mathcal{C}\rightarrow \mathcal{C}$ and two natural transformations, the \emph{unit} $\eta\colon{\rm Id}_\mathcal{C}\rightarrow M$  and the \emph{multiplication} $\mu \colon M^2\rightarrow M$, subject to the relations $\mu \circ M\mu =\mu \circ \mu M$ and $\mu\circ M\eta={\rm Id}_M=\mu\circ \eta M$. We sometimes suppress the unit and the multiplication, and denote the monad simply by $M$.

Monads arise naturally in adjoint pairs. Assume that $F\colon \mathcal{C}\rightarrow \mathcal{D}$ is a functor which admits a right adjoint $U\colon \mathcal{D}\rightarrow \mathcal{C}$. We denote by $\eta\colon {\rm Id}_\mathcal{C}\rightarrow UF$ the unit and $\epsilon\colon FU\rightarrow {\rm Id}_\mathcal{D}$ the counit; they satisfy $\epsilon F\circ F\eta ={\rm Id}_F$ and $U\epsilon\circ \eta U={\rm Id}_U$. We denote this adjoint pair on $\mathcal{C}$ and $\mathcal{D}$ by the quadruple $(F, U; \eta, \epsilon)$. We suppress the unit and the counit when they are clear from the context.

Each adjoint pair $(F, U;\eta, \epsilon)$ on two categories $\mathcal{C}$ and $\mathcal{D}$ defines a monad $(M, \eta, \mu)$ on $\mathcal{C}$, where  $M=UF\colon \mathcal{C}\rightarrow \mathcal{C}$ and $\mu=U\epsilon F\colon M^2=UFUF\rightarrow U{\rm Id}_\mathcal{D}F=M$. The resulting monad
$(M, \eta, \mu)$ on $\mathcal{C}$ is said to be \emph{defined} by the adjoint pair $(F, U; \eta, \epsilon)$. Indeed, as recalled now every monad is defined by some adjoint pair.

For a monad $(M, \eta, \mu)$ on $\mathcal{C}$, an $M$-\emph{module} is a pair $(X, \lambda)$ consisting of an object $X$ in $\mathcal{C}$ and a morphism $\lambda \colon M(X)\rightarrow X$ subject to the conditions $\lambda\circ M\lambda =\lambda\circ \mu_X$ and $\lambda\circ \eta_X={\rm Id}_X$;  the object $X$ is said to be the \emph{underlying object} of the module. A morphism  $f\colon (X, \lambda)\rightarrow (X', \lambda')$  between two $M$-modules is a morphism $f\colon  X\rightarrow X'$ in $\mathcal{C}$ satisfying $f\circ \lambda=\lambda'\circ M(f)$.  This gives rise to the category  $M\mbox{-Mod}_\mathcal{C}$ of $M$-modules; moreover, we have the \emph{forgetful functor} $U_M\colon M\mbox{-Mod}_\mathcal{C}\rightarrow \mathcal{C}$. The functor $U_M$ is faithful.

We observe that each object $X$ in $\mathcal{C}$ yields an $M$-module $F_M(X)=(M(X), \mu_X)$, the \emph{free $M$-module} generated by $X$. Indeed, this gives rise to the \emph{free module functor}  $F_M\colon \mathcal{C}\rightarrow M\mbox{-Mod}_\mathcal{C}$ sending $X$ to the free module $F_M(X)$, and a morphism $f\colon X\rightarrow Y$ to the morphism
$M(f)\colon F_M(X)\rightarrow F_M(Y)$.

We have the adjoint pair $(F_M, U_M; \eta, \epsilon_M)$ on $\mathcal{C}$ and $M\mbox{-Mod}_\mathcal{C}$, where for an
$M$-module $(X, \lambda)$, the counit $\epsilon_M$ is given such that $$(\epsilon_M)_{(X, \lambda)}=\lambda\colon F_M U_M(X,\lambda)=(M(X), \mu_X)\longrightarrow (X, \lambda).$$ We observe that this adjoint pair $(F_M, U_M; \eta, \epsilon_M)$ defines the given monad $(M, \eta, \mu)$.

The above adjoint pair $(F_M, U_M; \eta, \epsilon_M)$ enjoys the following universal property: for any adjoint pair $(F, U; \eta, \epsilon)$ on $\mathcal{C}$ and $\mathcal{D}$ that defines the monad $M$, there is a unique functor $$K\colon \mathcal{D}\longrightarrow M\mbox{-Mod}_\mathcal{C}$$ satisfying $KF=F_M$ and $U_MK=U$; see \cite[VI.3]{McL}. This unique functor $K$ will be referred as the \emph{comparison functor} associated to the  adjoint pair $(F, U; \eta, \epsilon)$. For its construction, we have
 \begin{align}\label{equ:K}
 K(D)=(U(D), U\epsilon_D), \quad K(f)=U(f)
\end{align}
for an object $D$ and a morphism $f$ in $\mathcal{D}$. Here, we observe that $M=UF$ and that $(U(D), U\epsilon_D)$ is an $M$-module.

Following \cite[VI.3]{McL} the adjoint pair $(F, U)$ is \emph{monadic} (\emph{resp.} \emph{strictly monadic}) provided that the associated comparison functor $K\colon \mathcal{D}\rightarrow M\mbox{-Mod}_\mathcal{C}$ is an equivalence (\emph{resp.} an isomorphism). In these cases, we might identify $\mathcal{D}$ with $M\mbox{-Mod}_\mathcal{C}$.

The famous monadicity theorem of Beck gives intrinsic characterizations on (strictly) monadic adjoint pairs; see \cite[VI.7]{McL}. However, the conditions in Beck's monadicity theorem on coequalizers seem rather technical.

 In the following result, we formulate a version of Beck's monadicity theorem for abelian categories, which is quite convenient for applications.

\begin{thm}\label{thm:Beck}
Let $F\colon \mathcal{A}\rightarrow \mathcal{B}$ be an additive functor on two abelian categories. Assume that $F$ admits a right adjoint $U\colon \mathcal{B}\rightarrow \mathcal{A}$ which is exact. Denote by $M$ the defined monad on $\mathcal{A}$. Then the comparison functor $K\colon \mathcal{B}\rightarrow M\mbox{-{\rm Mod}}_\mathcal{A}$ is an equivalence if and only if the functor $U\colon \mathcal{B}\rightarrow \mathcal{A}$ is faithful.
\end{thm}

The ``only if" part of  Theorem \ref{thm:Beck} is trivial, since $U=U_MK$ and the forgetful functor $U_M$ is faithful. We will present two proofs for the ``if" part. The first one is an application of the original version of Beck's monadicity theorem, while the second one is more self-contained and seems to be of independent interest.

 The following observation is well known.

\begin{lem}\label{lem:cok}
Let $F\colon \mathcal{A}\rightarrow \mathcal{B}$ be an exact functor on two abelian categories, which is faithful.  Then $F$ preserves and reflects cokernels, that is,  for two morphisms $f\colon X\rightarrow Y$ and $c\colon Y\rightarrow C$ in $\mathcal{A}$, $c$ is a cokernel of $f$ if and only if $F(c)$ is a cokernel of $F(f)$.
\end{lem}

\begin{proof}
The ``only if" part follows from the right exactness of $F$.

 For the ``if" part, we assume that $F(c)$ is a cokernel of $F(f)$. In particular, $F(c\circ f)=0$, and thus $c\circ f=0$ since $F$ is faithful. By the exactness of $F$, we have the following fact: for any complex $\xi\colon X\rightarrow Y\rightarrow Z$ in $\mathcal{A}$ with cohomology $H$ at $Y$,  the complex $F(\xi)\colon F(X)\rightarrow F(Y)\rightarrow F(Z)$ in $\mathcal{B}$ has cohomology $F(H)$ at $F(Y)$. Since $F$ is faithful, we infer that the complex $F(\xi)$ is exact at $F(Y)$, that is, $F(H)\simeq 0$, if and only if $H\simeq 0$, that is, the complex $\xi$ is exact at $Y$. We apply this fact to the complexes $X\stackrel{f}\rightarrow Y\stackrel{c}\rightarrow C$ and $Y\stackrel{c}\rightarrow C\rightarrow 0$. We infer that $c$ is a  cokernel of $f$.
\end{proof}

\vskip 5pt

\noindent \emph{The first proof of the ``if" part of Theorem \ref{thm:Beck}.}\quad Recall that a coequalizer of two parallel morphisms in an additive category equals a cokernel of their difference.  Then  we apply Lemma \ref{lem:cok} to the functor $U$, and infer that $U$ preserves and reflects coequalizers. Then the comparison functor $K$ is an equivalence by the ``equivalence" version of Beck's monadicity theorem; see \cite[VI.7, Exercise 6]{McL}. \hfill $\square$

\vskip 5pt

The second proof is an application of a slightly more general result, which seems to be standard in relative homological algebra.

 Let $\mathcal{X}$ be a full subcategory of a category $\mathcal{C}$. Let $C$ be an object in $\mathcal{C}$. Recall that a \emph{right $\mathcal{X}$-approximation} of $C$ is a morphism $f\colon X\rightarrow C$ with $X\in \mathcal{X}$ such that any morphism $t\colon T\rightarrow C$ with $T\in \mathcal{X}$ factors through $f$, that is, there exists a morphism $t'\colon T\rightarrow X$ with $t=f\circ t'$; compare \cite[Section 3]{AS}. In case that $\mathcal{C}$ is abelian,
an \emph{$\mathcal{X}$-presentation} of $C$  means an exact sequence $X_1\stackrel{g}\rightarrow X_0\stackrel{f}\rightarrow C\rightarrow 0$ such that $f$ is a right $\mathcal{X}$-approximation of $C$ and that the induced morphism $\bar{g}\colon X_1\rightarrow {\rm Ker}\; f$  is a right $\mathcal{X}$-approximation of ${\rm Ker}\; f$.

For a functor $F\colon \mathcal{C}\rightarrow \mathcal{D}$ and a full subcategory $\mathcal{X}\subseteq \mathcal{C}$,  we denote by $F(\mathcal{X})$ the full subcategory of $\mathcal{D}$ consisting of objects that are isomorphic to  $F(X)$ for some object $X$ in $\mathcal{X}$. In particular, we write ${\rm Im}\; F=F(\mathcal{C})$, the \emph{essential image} of $F$.

\begin{exm}\label{exm:1}
Let $(F, U; \eta, \epsilon)$ be an adjoint pair on two categories $\mathcal{C}$ and $\mathcal{D}$. Then for any object $D$ in $\mathcal{D}$, the counit $\epsilon_D\colon FU(D)\rightarrow D$ is a right ${\rm Im}\; F$-approximation. Indeed, any morphism $t\colon F(C)\rightarrow D$ factors through $\epsilon_D$ as $t=\epsilon_D\circ F(U(t)\circ \eta_C)$.

Let us state a special case explicitly. For a monad $M=(M, \eta, \mu)$ on $\mathcal{C}$, we apply the above to the adjoint pair $(F_M, U_M; \eta, \epsilon_M)$. Then for any $M$-module $(X, \lambda)$, the counit
\begin{align}\label{equ:epi}
(\epsilon_M)_{(X, \lambda)}=\lambda\colon F_MU_M(X, \lambda)=(M(X), \mu_X)\longrightarrow (X, \lambda)
\end{align}
is a right ${\rm Im}\; F_M$-approximation  of $(X, \lambda)$ in $M\mbox{-{\rm Mod}}_\mathcal{C}$; it is an epimorphism, since $\lambda\circ \eta_X={\rm Id}_X$.
\end{exm}

The following general result seems to be standard in relative homological algebra, whose proof is also standard.

\begin{prop}\label{prop:equi}
Let $F\colon \mathcal{A}\rightarrow \mathcal{B}$ be a right exact functor on two abelian categories. Let $\mathcal{X}$ be a full subcategory of $\mathcal{A}$. We assume that the following conditions are satisfied.
\begin{enumerate}
\item[(i)] For any object $C$ in $\mathcal{A}$, there exists an epimorphism $f\colon X\rightarrow C$ which is a right $\mathcal{X}$-approximation such that $F(f)$ is a right $F(\mathcal{X})$-approximation of $F(C)$.
\item[(ii)] The restricted functor $F|_\mathcal{X}\colon \mathcal{X}\rightarrow \mathcal{B}$ is fully faithful.
\end{enumerate}
Then the functor $F\colon \mathcal{A}\rightarrow \mathcal{B}$ is fully faithful. Assume that in addition the following condition is satisfied.
\begin{enumerate}
\item[(iii)] For any object $B$ in $\mathcal{B}$, there is an epimorphism $F(A)\rightarrow B$ for some object $A$ in $\mathcal{A}$.
\end{enumerate}
Then the functor $F\colon \mathcal{A} \rightarrow \mathcal{B}$ is an equivalence.
\end{prop}

\begin{proof}
Let $A, A'$ be two objects in $\mathcal{A}$. By the right exactness of $F$ and (i),  we have two  $\mathcal{X}$-presentations $\xi\colon X_1\stackrel{d}\rightarrow X_0\stackrel{p}\rightarrow A\rightarrow 0$  and $\xi'\colon X'_1\stackrel{d'}\rightarrow X'_0\stackrel{p'}\rightarrow A'\rightarrow 0$ such that $F(\xi)$ and $F(\xi')$ are $F(\mathcal{X})$-presentations of $F(A)$ and $F(A')$, respectively. To show that $F$ is full, take a morphism $t\colon F(A)\rightarrow F(A')$. Since $F(\xi')$ is an $F(\mathcal{X})$-presentation, we have the following commutative diagram
\[\xymatrix{
F(X_1) \ar@{.>}[d]^-{t_1}\ar[r]^{F(d)} & F(X_0) \ar@{.>}[d]^-{t_0} \ar[r]^{F(p)} & F(A)\ar[d]^-{t} \ar[r] & 0\\
F(X'_1) \ar[r]^-{F(d')} & F(X'_0) \ar[r]^{F(p')} & F(A') \ar[r] & 0.
}\]
By (ii), there exist morphisms $s_i\colon X_i\rightarrow X'_i$ such that $F(s_i)=t_i$ and $s_0\circ d=d'\circ s_1$. The identity $s_0\circ d=d'\circ s_1$ implies the existence of a morphism $s\colon A\rightarrow A'$ such that $s \circ p=p'\circ s_0$. It follows that $F(s)=t$ by comparing commutative diagrams.

The faithfulness of $F$ follows by a converse argument. Indeed, given a morphism  $s\colon A\rightarrow A'$ with $F(s)=0$, we have morphisms $s_i\colon X_i\rightarrow X'_i$ such that $s \circ p=p'\circ s_0$ and $s_0\circ d=d'\circ s_1$; here, we use the $\mathcal{X}$-presentation $\xi'$. Then in the above diagram, $t=F(s)=0$ and thus $t_0=F(s_0)$ factors through $F(d')$. By (ii), $s_0$ factors through $d'$, and thus $s\circ p=p'\circ s_0=0$. We infer that $s=0$, since $p$ is epic.

We assume that (iii) is satisfied. Then any object $B$ fits into an exact sequence $F(A')\stackrel{g}\rightarrow F(A)\rightarrow B\rightarrow 0$. By the fully faithfulness of $F$, there exists a morphism $s\colon A'\rightarrow A$ satisfying $F(s)=g$. It follows that $F({\rm Cok}\; s)\simeq B$. This proves the denseness of $F$, and thus $F$ is an equivalence.
\end{proof}

\vskip 5pt

\noindent \emph{The second proof of the ``if" part of Theorem \ref{thm:Beck}.}\quad Since $F$ has a right adjoint, it is right exact. Then the endofunctor $M=UF$ on $\mathcal{A}$ is right exact. It follows that the category $M\mbox{-Mod}_\mathcal{A}$ is abelian; indeed, a sequence of $M$-modules is exact if and only if the corresponding sequence of the underlying objects in $\mathcal{A}$ is exact.

Recall from (\ref{equ:K}) that $K(B)=(U(B), U\epsilon_B)$ and that $K(f)=U(f)$ for an object $B$ and a  morphism $f$ in $\mathcal{B}$. Since the functor $U$ is exact and faithful,  it follows that the comparison functor $K\colon \mathcal{B}\rightarrow M\mbox{-Mod}_\mathcal{A}$ is exact and faithful. Consider $\mathcal{X}={\rm Im}\; F$ as a full subcategory of $\mathcal{B}$. We claim that the pair $K$ and $\mathcal{X}$ satisfy the three conditions in Proposition \ref{prop:equi}. Then we infer that $K$ is an equivalence.

For the claim, we verify the conditions (i)-(iii). For (i), take any object $B$ in $\mathcal{B}$ and consider the counit $\epsilon_B\colon FU(B)\rightarrow B$, which is a right $\mathcal{X}$-approximation by Example \ref{exm:1}. We observe that $K(\epsilon_B)$ equals $(\epsilon_M)_{K(B)}$, the counit $\epsilon_M$ applied on the $M$-module $K(B)$. In particular, by Example \ref{exm:1} $K(\epsilon_B)$ is a right ${\rm Im}\; F_M$-approximation of $K(B)$ in $M\mbox{-Mod}_\mathcal{A}$, which is an epimorphism. Recall that $F_M=KF$ and thus ${\rm Im}\; F_M=K(\mathcal{X})$. Hence, the epimorphism $K(\epsilon_B)$ is a right $K(\mathcal{X})$-approximation of $K(B)$. By Lemma \ref{lem:cok} the functor $K$ reflects epimorphisms. It follows that $\epsilon_B$ is an epimorphism. This is the required morphism in (i).

The condition (ii) is well known, since the restriction of $K$ on $\mathcal{X}={\rm Im}\; F$ is fully faithful; see \cite[Lemma 3.3]{Ch} and compare \cite[VI.3 and VI.5]{McL}. The condition (iii) follows from the epimorphism (\ref{equ:epi}) for any $M$-module $(X, \lambda)$, since $F_MU_M=KFU_M$ and thus $(M(X), \mu_X)=KF(X)$. Set $B=F(X)$. In particular, we have by (\ref{equ:epi}) a required epimorphism $K(B)\rightarrow (X, \lambda).$
\hfill $\square$

\vskip 5pt

\section{Equivariant objects as modules}

In this section, we recall the notions of a group action on a category and the category of equivariant objects. We give a direct proof of the fact that in the additive case,  the category of equivariant objects is isomorphic to the category of modules over a certain monad.

Let $G$ be an arbitrary group. We write $G$ multiplicatively and denote its unit by $e$. Let $\mathcal{C}$ be an arbitrary category.

The notion of a group action on a category is well known; compare \cite{De,RR,DGNO}.   An \emph{action} of $G$ on $\mathcal{C}$ consists of the data $\{F_g, \varepsilon_{g, h}|\; g, h\in G\}$, where each $F_g\colon \mathcal{C}\rightarrow \mathcal{C}$ is an auto-equivalence and each $\varepsilon_{g, h}\colon F_gF_h\rightarrow F_{gh}$ is a natural isomorphism such that a $2$-cocycle condition holds, that is,
\begin{align}\label{equ:2-coc}
\varepsilon_{gh, k}\circ \varepsilon_{g,h}F_k=\varepsilon_{g, hk}\circ F_g\varepsilon_{h,k}
\end{align}
for all $g, h, k\in G$. We observe that there exists a unique natural isomorphism $u\colon F_e\rightarrow {\rm Id}_\mathcal{C}$, called the \emph{unit} of the action, satisfying $\varepsilon_{e,e}=F_eu$; moreover, we have $F_{e}u=uF_e$ by (\ref{equ:2-coc}).

The given action is \emph{strict} provided that each $F_g$ is an automorphism and each isomorphism $\varepsilon_{g, h}$ is the identity, in which case the unit is also the identity. Therefore, a strict action coincides with  a group homomorphism from $G$ to the automorphism group of $\mathcal{C}$.

 Let $G$ act on $\mathcal{C}$.  A \emph{$G$-equivariant object} in $\mathcal{C}$ is a pair $(X, \alpha)$, where $X$ is an object in $\mathcal{C}$ and $\alpha$ assigns for each $g\in G$ an isomorphism $\alpha_g\colon X\rightarrow F_g(X)$ subject to the relations
 \begin{align}\label{equ:rel}
 \alpha_{gg'}=(\varepsilon_{g,g'})_X \circ F_g(\alpha_{g'}) \circ \alpha_g.\end{align}
 These relations imply that $\alpha_e=u^{-1}_X$. A morphism $\theta\colon (X, \alpha)\rightarrow (Y, \beta)$ between  two $G$-equivariant objects is a morphism $\theta\colon X\rightarrow Y$ in $\mathcal{C}$ such that
  $\beta_g\circ \theta=F_g(\theta)\circ \alpha_g$ for all $g\in G$. This gives rise to
   the category $\mathcal{C}^G$ of $G$-equivariant objects, and the \emph{forgetful functor}
   $U\colon \mathcal{C}^G\rightarrow \mathcal{C}$ defined by $U(X, \alpha)=X$. The process forming the category $\mathcal{C}^G$ of equivariant objects is known as the \emph{equivariantization} with respect to the group action; see \cite{DGNO}.

In what follows, we assume that the group $G$ is finite and that $\mathcal{C}$ is an additive category. In this case the forgetful functor $U$
 admits a left adjoint $F\colon \mathcal{C}\rightarrow \mathcal{C}^G$, which is known as the \emph{induction functor};  see \cite[Lemma 4.6]{DGNO}.
The functor $F$ is defined as follows: for an object $X$, set $F(X)=(\oplus_{h\in G} F_h(X), \varepsilon)$, where for each $g\in G$, the isomorphism ${\rm \varepsilon}_g \colon \oplus_{h\in G} F_h(X)\rightarrow F_g(\oplus_{h\in G} F_h(X))$ is diagonally induced by the isomorphism $(\varepsilon_{g, g^{-1}h})_X^{-1}\colon F_h(X)\rightarrow F_g(F_{g^{-1}h}(X))$. Here, to verify that $F(X)$ is indeed an equivariant object, we need the $2$-cocycle condition (\ref{equ:2-coc}). The functor $F$ sends a morphism $\theta\colon X\rightarrow Y$ to $F(\theta)=\oplus_{h\in G} F_h(\theta)\colon F(X)\rightarrow F(Y)$.

For an object $X$ in $\mathcal{C}$ and an object $(Y, \beta)$ in $\mathcal{C}^G$,
  a morphism $F(X)\rightarrow (Y, \beta)$ is of the form $\sum_{h\in G}\theta_h\colon \oplus_{h\in G} F_h(X)\rightarrow Y$
  satisfying $F_g(\theta_h)=\beta_g\circ \theta_{gh}\circ (\varepsilon_{g, h})_X$ for any $g, h\in G$.  The adjunction of $(F, U)$ is given by the
  following natural isomorphism
\begin{align*}
 {\rm Hom}_{\mathcal{C}^G} (F(X), (Y, \beta))\stackrel{\sim}\longrightarrow {\rm Hom}_\mathcal{C}(X, U(Y, \beta))
 \end{align*}
 sending $\sum_{h\in G} \theta_h$ to $\theta_e \circ u_X^{-1}$. The corresponding unit $\eta\colon {\rm Id}_\mathcal{C}\rightarrow UF$ is
 given such that $\eta_X=(u_X^{-1}, 0, \cdots, 0)^t$, where `t' denotes the transpose;
 the counit $\epsilon\colon FU\rightarrow {\rm Id}_{\mathcal{C}^G}$ is given such that
  $\epsilon_{(Y, \beta)}=\sum_{h\in G} \beta_h^{-1}$.

Let us compute the monad $M=(UF, \eta, \mu)$ defined by the  adjoint pair $(F, U; \eta, \epsilon)$. This monad is said to be \emph{defined} by the group action.  The endofunctor $M\colon \mathcal{C}\rightarrow \mathcal{C}$ is given by $M(X)=\oplus_{h\in G} F_h(X)$ and $M(\theta)=\oplus_{h\in G} F_h(\theta)$ for a morphism $\theta$ in $\mathcal{C}$. The multiplication $\mu\colon M^2\rightarrow M$ is given by
 $$\mu_X=U\epsilon_{F(X)}\colon M^2(X)=\oplus_{h, g\in G} F_hF_g(X) \longrightarrow M(X)=\oplus_{h\in G} F_h(X)$$
 with the property that the corresponding entry
  $F_hF_g(X)\rightarrow F_{h'}(X)$ is $\delta_{hg, h'}(\varepsilon_{h, g})_X$; here, $\delta$ is the Kronecker symbol.

The following result shows that the category of equivariant objects is isomorphic to the category of $M$-modules. Roughly speaking, equivariant objects are modules. We mention that the result is an immediate consequence of Beck's monadicity theorem in \cite[VI.7]{McL}; see also \cite[Lemma 4.3]{Ch} and compare \cite[Proposition 3.10]{El2014}. We provide a direct proof for completeness.

 \begin{prop}\label{prop:monadic}
 Let $\mathcal{C}$ be an additive category and $G$ be a finite group acting on $\mathcal{C}$. Keep the notation as above.
 Then the adjoint pair $(F,U; \eta, \epsilon)$ is strictly monadic, that is, the associated comparison functor $K\colon \mathcal{C}^G\rightarrow M\mbox{-{\rm Mod}}_\mathcal{C}$ is an isomorphism of categories.
 \end{prop}

\begin{proof}
We have just computed explicitly the monad $M$.   Let us take a closer look at $M$-modules. An $M$-module is a pair $(X, \lambda)$ with $\lambda=\sum_{h\in G}\lambda_h\colon M(X)=\oplus_{h\in G} F_h(X)\rightarrow X$.
  The condition $\lambda\circ \eta_X={\rm Id}_X$ is equivalent to $\lambda_e=u_X$,
  and $\lambda\circ M(\lambda)=\lambda\circ \mu_X$ is equivalent to
   $\lambda_{hg}\circ \varepsilon_{h,g}=\lambda_h\circ F_h(\lambda_g)$ for any $h, g\in G$. Hence, if we
   set $\alpha_h\colon X\rightarrow F_h(X)$ to be $(\lambda_h)^{-1}$ and compare (\ref{equ:rel}), we obtain an object $(X, \alpha)\in \mathcal{C}^G$.
    Roughly speaking, the maps $\lambda_h$'s carry the same information as $\alpha_h$'s.

Recall  from (\ref{equ:K}) that the associated comparison functor $K\colon \mathcal{C}^G\rightarrow M\mbox{-Mod}_\mathcal{C}$  is constructed such that $K(X, \alpha)=(U(X, \alpha), U\epsilon_{(X, \alpha)})$, which equals the $M$-module $(X, \lambda)$ with $\lambda_h=(\alpha_h)^{-1}$ by the explicit form of the counit $\epsilon$. It follows immediately that $K$ induces a bijection on objects, and is fully faithful. We infer that $K$ is  an isomorphism of categories.
\end{proof}

\begin{rem}\label{rem:monadic}
Proposition \ref{prop:monadic} can be extended slightly. Let $G$ be a group whose cardinality $|G|$ might be infinite. Assume that the additive category $\mathcal{C}$ has coproducts with any index set of cardinality less or equal to $|G|$. For a $G$-action on $\mathcal{C}$, we define the induction functor $F$ and the monad $M$ as above by replacing finite coproducts by coproducts indexed by a possibly infinite index set. The same argument proves that the comparison functor $K\colon \mathcal{C}^G\rightarrow M\mbox{-{\rm Mod}}_\mathcal{C}$ is an isomorphism of categories.
\end{rem}

\section{Exact monads and quotient abelian categories}

 In this section, we give the first application of Theorem \ref{thm:Beck},  which states that the formation of the module category of  an exact monad is compatible with quotient abelian categories by Serre subcategories. This result applies to the category of equivariant objects in an abelian category.

 \subsection{}

 Let $\mathcal{A}$ be an abelian category. A \emph{Serre subcategory} $\mathcal{N}$ of $\mathcal{A}$ is by definition a full subcategory which is closed under subobjects, quotient objects and extensions. In other words, for an exact sequence $0\rightarrow X\rightarrow Y\rightarrow Z\rightarrow 0$ in $\mathcal{A}$, $Y$ lies in $\mathcal{N}$ if and only if both $X$ and $Z$ lie in $\mathcal{N}$. It follows that a Serre subcategory $\mathcal{N}$ is abelian and the inclusion functor $\mathcal{N}\rightarrow \mathcal{A}$ is exact.

 For a Serre subcategory $\mathcal{N}$ of $\mathcal{A}$, we denote by $\mathcal{A}/\mathcal{N}$ the \emph{quotient abelian category}; see \cite{Ga62}. The objects of $\mathcal{A}/\mathcal{N}$ are the same as $\mathcal{A}$, and for two objects $X$ and $Y$, a morphism in $\mathcal{A}/\mathcal{N}$ is represented by a morphism $X'\rightarrow Y/Y'$, where $X'\subseteq X$ and $Y'\subseteq Y$ are subobjects with both $X/X'$ and $Y'$ in $\mathcal{N}$. The quotient functor $q\colon \mathcal{A}\rightarrow \mathcal{A}/\mathcal{N}$ sends an object $X$ to $X$, a morphism $f\colon X\rightarrow Y$ to the morphism $q(f)$,  which is represented by $f\colon X\rightarrow Y$; here,  the corresponding $X'$  and $Y'$ are $X$ and $0$, respectively. The functor $q$ is exact with its essential kernel ${\rm Ker}\; q=\mathcal{N}$. In particular, for a morphism $f\colon X\rightarrow Y$ in $\mathcal{A}$, $q(f)=0$ if and only if its image ${\rm Im}\; f$ lies in $\mathcal{N}$.

Let $F\colon \mathcal{A}\rightarrow \mathcal{A}'$ be an exact functor. Then the essential kernel ${\rm Ker}\; F$ is a Serre subcategory of $\mathcal{A}$, and $F$ induces a unique exact functor $F'\colon \mathcal{A}/{{\rm Ker}\; F}\rightarrow \mathcal{A}'$. We say that $F$ is a \emph{quotient functor} provided that $F'$ is an equivalence.

Let $F\colon \mathcal{A}\rightarrow \mathcal{A}'$ be an exact functor. Assume that $\mathcal{N}\subseteq \mathcal{A}$ and $\mathcal{N}'\subseteq \mathcal{A}'$ are Serre subcategories such that $F(\mathcal{N})\subseteq \mathcal{N}'$. Then there is a uniquely induced exact functor $\bar{F}\colon \mathcal{A}/\mathcal{N}\rightarrow \mathcal{A}'/{\mathcal{N}'}$.

\begin{lem}\label{lem:quoinduced}
Keep the notation as above. Assume that $F\colon \mathcal{A}\rightarrow \mathcal{A}'$ is a quotient functor. Then the induced functor $\bar{F}\colon \mathcal{A}/\mathcal{N}\rightarrow \mathcal{A}'/{\mathcal{N}'}$ is also a quotient functor.
\end{lem}

\begin{proof}
Denote by $\mathcal{C}$ the inverse image of $\mathcal{N}'$ under $F$. Then $\mathcal{C}$ is a Serre subcategory of $\mathcal{A}$ containing $\mathcal{N}$ and ${\rm Ker}\; F$. The essential kernel ${\rm Ker}\; \bar{F}$ of $\bar{F}$ equals $\mathcal{C}/\mathcal{N}$.

We observe the following isomorphisms of abelian categories
$$(\mathcal{A}/\mathcal{N})/(\mathcal{C}/\mathcal{N})\simeq \mathcal{A}/\mathcal{C}\simeq (\mathcal{A}/{{\rm Ker}\; F})/(\mathcal{C}/{{\rm Ker}\; F}).$$
By the assumption, we have the equivalence $F'\colon \mathcal{A}/{{\rm Ker}\; F}\stackrel{\sim}\longrightarrow \mathcal{A}'$, which induces an equivalence $(\mathcal{A}/{{\rm Ker}\; F})/(\mathcal{C}/{{\rm Ker}\; F}) \stackrel{\sim}\longrightarrow \mathcal{A}'/\mathcal{N}'$. Consequently, we have the required equivalence $(\mathcal{A}/\mathcal{N})/{\rm Ker}\; \bar{F}\stackrel{\sim}\longrightarrow \mathcal{A}'/\mathcal{N}'$.
\end{proof}

 Let $M\colon \mathcal{A}\rightarrow \mathcal{A}$ be an \emph{exact} monad on $\mathcal{A}$, that is, $M=(M, \eta, \mu)$ is a monad on $\mathcal{A}$ and the endofunctor $M$ is exact. Recall that the category $M\mbox{-Mod}_\mathcal{A}$ of $M$-modules is abelian, where a sequence of $M$-modules is exact if and only if the corresponding sequence of underlying objects in $\mathcal{A}$ is exact. It follows that both the free module functor $F_M\colon \mathcal{A}\rightarrow M\mbox{-Mod}_\mathcal{A}$ and the forgetful functor $U_M\colon M\mbox{-Mod}_\mathcal{A}\rightarrow \mathcal{A}$ are exact.

 Let $\mathcal{N}\subseteq \mathcal{A}$ be a Serre subcategory such that $M(\mathcal{N})\subseteq \mathcal{N}$. Then we have the induced endofunctor $\bar{M}\colon \mathcal{A}/\mathcal{N}\rightarrow \mathcal{A}/\mathcal{N}$; moreover, the natural transformations $\eta$ and $\mu$ induce natural transformations $\bar{\eta}\colon {\rm Id}_{\mathcal{A}/\mathcal{N}}\rightarrow \bar{M}$ and $\bar{\mu}\colon \bar{M}^2\rightarrow \bar{M}$. Indeed, we obtain a monad $\bar{M}=(\bar{M}, \bar{\eta}, \bar{\mu})$ on $\mathcal{A}/\mathcal{N}$, called the \emph{induced monad} of $M$.

 We will need the following standard fact.

\begin{lem}\label{lem:quotientAbel}
Let $F\colon \mathcal{A}\rightarrow \mathcal{A}'$ be an exact functor between two abelian categories, which has an exact right
adjoint $U\colon \mathcal{A}'\rightarrow \mathcal{A}$. Assume that $\mathcal{N}\subseteq \mathcal{A}$ and $\mathcal{N}'\subseteq \mathcal{A}'$ are Serre subcategories such that $F(\mathcal{N})\subseteq \mathcal{N}'$ and $U(\mathcal{N}')\subseteq \mathcal{N}$. Then the following statements hold.
\begin{enumerate}
\item The induced functor $\bar{F}\colon \mathcal{A}/\mathcal{N}\rightarrow \mathcal{A}'/{\mathcal{N}'}$ is left adjoint to the induced functor $\bar{U}\colon \mathcal{A}'/\mathcal{N}'\rightarrow \mathcal{A}/\mathcal{N}$.
    \item The monad on $\mathcal{A}/\mathcal{N}$ defined by the adjoint pair $(\bar{F}, \bar{U})$ coincides with the induced monad of the one on $\mathcal{A}$ defined by the adjoint pair $(F, U)$.
        \item Denote by $U^{-1}(\mathcal{N})$ the inverse image of $\mathcal{N}$. Assume that $U^{-1}(\mathcal{N})=\mathcal{N}'$. Then the induced functor $\bar{U}$ is faithful.
 \end{enumerate}
\end{lem}

\begin{proof}
Denote the given adjoint pair by $(F, U; \eta, \epsilon)$ and its defined monad by $(M=UF, \eta, \mu)$.
We observe that the unit $\eta$ (\emph{resp.},
the counit $\epsilon$) induces naturally a natural transformation
$\bar{\eta}\colon {\rm Id}_{\mathcal{A}/\mathcal{N}}
\rightarrow \bar{U}\bar{F}$ (\emph{resp.}, $\bar{\epsilon}\colon
\bar{F} \bar{U} \rightarrow {\rm Id}_{\mathcal{A}'/{\mathcal{N}'}}$). Then we apply
\cite[IV.1, Theorem 2(v)]{McL} to deduce the adjoint pair  $(\bar{F}, \bar{U}; \bar{\eta}, \bar{\epsilon})$. It follows that the monad defined by this adjoint pair coincides with $(\bar{M}, \bar{\eta}, \bar{\mu})$, the induced monad of $M$.

For (3), take any morphism $\theta\colon X\rightarrow Y$ in $\mathcal{A}'/\mathcal{N}'$ with $\bar{U}(\theta)=0$. We assume that $\theta$ is represented by a morphism $f\colon X'\rightarrow Y/Y'$, where $X'\subseteq X$ and $Y'\subseteq Y$ are subobjects satisfying that both $X/X'$ and $Y'$ lie in $\mathcal{N}'$. Then $\bar{U}(\theta)$ is represented by $U(f)\colon U(X')\rightarrow U(Y/Y')$, and thus $q(U(f))=0$ in $\mathcal{A}/\mathcal{N}$, or equivalently, the image ${\rm Im}\;U(f)$ lies in $\mathcal{N}$. We observe that ${\rm Im}\; U(f)\simeq U({\rm Im}\; f)$ and then ${\rm Im}\; f$ lies in $\mathcal{N}'$, since $U^{-1}(\mathcal{N})=\mathcal{N}'$. We infer that $\theta$ equals zero in $\mathcal{A}'/{\mathcal{N}'}$. We are done. \end{proof}

\subsection{}

 Let $\mathcal{A}$ be an abelian category, and $M\colon \mathcal{A} \rightarrow \mathcal{A}$ be an exact monad. We will apply Lemma  \ref{lem:quotientAbel} to the adjoint pair $(F_M, U_M)$.

 Let $\mathcal{N}\subseteq \mathcal{A}$ be a Serre subcategory with $M(\mathcal{N})\subseteq \mathcal{N}$. Then we have the restricted
 exact monad $M\colon \mathcal{N}\rightarrow \mathcal{N}$ and the category $M\mbox{-Mod}_\mathcal{N}$ of $M$-modules in $\mathcal{N}$. We view $M\mbox{-Mod}_\mathcal{N}$ as a full subcategory of $M\mbox{-Mod}_\mathcal{A}$, in other words, an $M$-module $(X, \lambda)$ lies in $M\mbox{-Mod}_\mathcal{N}$ if and only if the underlying object $X$ lies in $\mathcal{N}$. We observe that $M\mbox{-Mod}_\mathcal{N}\subseteq M\mbox{-Mod}_\mathcal{A}$ is a Serre subcategory.

Consider the induced monad $\bar{M}$ on $\mathcal{A}/\mathcal{N}$. We observe that an $M$-module $(X, \lambda)$ yields an $\bar{M}$-module $(X, q(\lambda))$ on $\mathcal{A}/\mathcal{N}$, where $q(\lambda)\colon \bar{M}(X)=qM(X)\rightarrow X$. This gives rise to an exact functor $$\Phi\colon M\mbox{-Mod}_\mathcal{A}\longrightarrow \bar{M}\mbox{-Mod}_{\mathcal{A}/\mathcal{N}}, \quad (X, \lambda)\mapsto (X, q(\lambda)).$$
The functor $\Phi$ sends a morphism $\theta$ to $q(\theta)$. We observe that $\Phi$ vanishes on the Serre subcategory  $M\mbox{-Mod}_\mathcal{N}$, and induces an exact functor $\bar{\Phi}\colon M\mbox{-{\rm Mod}}_\mathcal{A}/{M\mbox{-{\rm Mod}}_\mathcal{N}}\rightarrow \bar{M}\mbox{-{\rm Mod}}_{\mathcal{A}/\mathcal{N}}$.

 \begin{prop}\label{prop:app1}
 Let $M\colon \mathcal{A}\rightarrow \mathcal{A}$ be an exact monad on an abelian category $\mathcal{A}$, and let $\mathcal{N}\subseteq \mathcal{A}$ be a Serre subcategory satisfying $M(\mathcal{N})\subseteq \mathcal{N}$. Keep the notation as above. Then the induced functor
 $$\bar{\Phi}\colon M\mbox{-{\rm Mod}}_\mathcal{A}/{M\mbox{-{\rm Mod}}_\mathcal{N}}\stackrel{\sim}\longrightarrow \bar{M}\mbox{-{\rm Mod}}_{\mathcal{A}/\mathcal{N}}$$
  is an equivalence of categories.
 \end{prop}

 \begin{proof}
 Consider the free module functor $F_M\colon \mathcal{A}\rightarrow M\mbox{-{\rm Mod}}_\mathcal{A}$ and the forgetful functor $U_M\colon M\mbox{-{\rm Mod}}_\mathcal{A}\rightarrow \mathcal{A}$, both of which are exact. We observe that $F_M(\mathcal{N})\subseteq M\mbox{-{\rm Mod}}_\mathcal{N}$, and $U_M^{-1}(\mathcal{N})=M\mbox{-{\rm Mod}}_\mathcal{N}$. We apply Lemma \ref{lem:quotientAbel}(1) and (3), and obtain the adjoint pair $\overline{F_M}$ and $\overline{U_M}$ between quotient abelian categories $\mathcal{A}/\mathcal{N}$ and $M\mbox{-{\rm Mod}}_\mathcal{A}/{M\mbox{-{\rm Mod}}_\mathcal{N}}$, where $\overline{U_M}$ is faithful. Moreover, by Lemma \ref{lem:quotientAbel}(2) the monad defined by this adjoint pair coincides with the induced monad $\bar{M}$ of $M$. We apply Theorem \ref{thm:Beck} to the adjoint pair $(\overline{F_M}, \overline{U_M})$ and infer that the associated comparison functor $$K\colon M\mbox{-{\rm Mod}}_\mathcal{A}/{M\mbox{-{\rm Mod}}_\mathcal{N}}\longrightarrow \bar{M}\mbox{-Mod}_{\mathcal{A}/\mathcal{N}}$$ is an equivalence.

It remains to observe that $K=\bar{\Phi}$. Indeed, by (\ref{equ:K}) the comparison functor $K$  sends an $M$-module $(X, \lambda)$ to an $\bar{M}$-module $(\overline{U_M}(X, \lambda), \overline{U_M}(\bar{\epsilon}_M)_{(X, \lambda)})$, which equals $(X, q(\lambda))$. Hence, the functors $K$ and $\bar{\Phi}$  agree on objects. For the same reason, they agree on morphisms.
 \end{proof}

 We now apply Proposition \ref{prop:app1} to the category of equivariant objects. Let $G$ be a finite group which acts on an abelian category $\mathcal{A}$ by the data $\{F_g, \varepsilon_{g,h}\; |\; g,h\in G\}$. Then the category $\mathcal{A}^G$ of $G$-equivariant objects is abelian; indeed, a sequence of equivariant objects is exact if and only if the corresponding sequence of the underlying objects in $\mathcal{A}$ is exact.

 Let $\mathcal{N}\subseteq \mathcal{A}$ be a Serre subcategory which is \emph{invariant} under this action, that is, $F_g(\mathcal{N})\subseteq \mathcal{N}$ for any $g\in G$. Then the quotient category $\mathcal{A}/\mathcal{N}$ inherits a $G$-action, which is given by the data $\{\bar{F_g}, \bar{\varepsilon}_{g,h}\; |\; g,h\in G\}$. The quotient functor $q\colon \mathcal{A}\rightarrow \mathcal{A}/\mathcal{N}$ induces an exact functor
 $$\Psi\colon \mathcal{A}^G\longrightarrow (\mathcal{A}/\mathcal{N})^G.$$
More precisely, $\Psi$ sends a $G$-equivariant object $(X, \alpha)$ to $(X, q(\alpha))$, where $q(\alpha)_g=q(\alpha_g)\colon X\rightarrow \bar{F_g}(X)=qF_g(X)$ for each $g\in G$, and $\Psi$ sends a morphism $\theta\colon (X, \alpha)\rightarrow (Y, \beta)$ to $q(\theta)\colon (X, q(\alpha)) \rightarrow (Y, q(\beta))$. We observe that the functor $\Psi$ is exact and that its  essential kernel equals $\mathcal{N}^G$. Therefore, we have the induced functor $\bar{\Psi}\colon \mathcal{A}^G/{\mathcal{N}^G}\rightarrow (\mathcal{A}/\mathcal{N})^G$.

 \begin{cor}\label{cor:equiv}
 Let $G$ be a finite group acting on an abelian category $\mathcal{A}$, and let $\mathcal{N}\subseteq \mathcal{A}$ be a Serre subcategory which is invariant under the action. Keep the notation as above. Then the induced functor
 $$\bar{\Psi}\colon \mathcal{A}^G/{\mathcal{N}^G}\stackrel{\sim}\longrightarrow (\mathcal{A}/\mathcal{N})^G$$
 is an equivalence of categories.
 \end{cor}

\begin{proof}
Denote by $M$ the monad on $\mathcal{A}$ that is defined by the group action. Then $M$ is exact and by the invariance of $\mathcal{N}$, we have $M(\mathcal{N})\subseteq \mathcal{N}$. We observe that the induced monad $\bar{M}$ on $\mathcal{A}/\mathcal{N}$ coincides with the monad defined by the induced $G$-action on $\mathcal{A}/\mathcal{N}$.

By Proposition \ref{prop:monadic} we identify $\mathcal{A}^G$ with $M\mbox{-Mod}_\mathcal{A}$,  and   $(\mathcal{A}/\mathcal{N})^G$ with $\bar{M}\mbox{-Mod}_{\mathcal{A}/\mathcal{N}}$. Then the equivalence follows from Proposition \ref{prop:app1}. Here, one observes that the functor $\bar{\Phi}$ in  Proposition \ref{prop:app1} corresponds to the functor $\bar{\Psi}$.
\end{proof}

\section{Graded module categories and graded group rings}

In this section, we give the second application of Theorem \ref{thm:Beck}, and  prove that the equivarianzation of the graded module category over a graded ring with respect to a certain degree-shift action is equivalent to the graded module category over the same ring but with a coarser grading; this result is essentially due to  \cite[Theorem 6.4.1]{NV} in a different setup. On the other hand, we prove that the equivariantization of the graded module category over a graded ring with respect to a certain twisting action is equivalent to the graded module category over the same ring but with a refined grading. All modules are right modules.

\subsection{}

Let $G$ be an arbitrary group. Let $R=\oplus_{g\in G} R_g$ be a $G$-graded ring with a unit. Here, the unit $1_R$ lies in $R_e$, and the subgroups $R_g$ of $R$ are called the homogeneous components of degree $g$.

A $G$-graded $R$-module $X$ is an $R$-module with a decomposition $X=\oplus_{g\in G}{X_g}$ into homogeneous components $X_g$ such that $X_g.R_{g'}\subseteq X_{gg'}$, where the dot ``." means the right action.  An element $x$ in $X_g$ is said to be \emph{homogeneous of degree} $g$, denoted by $|x|=g$ or ${\rm deg}\; x=g$. The \emph{grading support} of $X$ is by definition ${\rm gsupp}(X)=\{g\in G\; |\; X_g\neq 0\}$, which is a subset of $G$.

We denote by ${\rm Mod}^G\mbox{-}R$ the abelian category of  $G$-graded $R$-modules, where the homomorphism between graded modules are module homomorphisms that preserve the degrees. More precisely, a homomorphism $f\colon X\rightarrow X'$ is an $R$-module homomorphism  satisfying $f(X_g)\subseteq X'_{g}$, and  we denote by $f_g\colon X_g\rightarrow X'_g$ its restriction to the corresponding homogeneous component. We denote by ${\rm mod}^G\mbox{-}R$ the full subcategory formed by finitely presented graded modules.

For a $G$-graded $R$-module $X$ and $g'\in G$, the \emph{shifted module} $X(g')$ is defined such that $X(g')=X$ as an ungraded $R$-module and that its grading is given by $X(g')_g=X_{g'g}$ for each $g\in G$. This gives rise to an automorphism $(g')\colon {\rm Mod}^G\mbox{-}R\rightarrow {\rm Mod}^G\mbox{-}R$, called the \emph{degree-shift functor}. For example, we consider $R(g')$ as a $G$-graded $R$-module with $1_R$ having degree $g'^{-1}$.  Indeed, the set $\{R(g)\;|\; g\in G\}$ is a set of projective generators in the category ${\rm Mod}^G\mbox{-}R$.

Each subgroup $G'\subseteq G$ has a strict action on ${\rm Mod}^G\mbox{-}R$ by assigning to each $g'\in G'$ the automorphism $F_{g'}=(g'^{-1})$. Such a $G'$-action on ${\rm Mod}^G\mbox{-}R$ is referred as a \emph{degree-shift action}; in this case, $G'$ acts also on ${\rm mod}^G\mbox{-}R$.

Let $\pi\colon G\rightarrow H$ be a homomorphism of groups. We define an $H$-graded ring $\pi_*(R)$ as follows: as an ungraded ring
$\pi_*(R)=R$, while its homogeneous component is given by $\pi_*(R)_h=\oplus_{g\in \pi^{-1}(h)} R_g$ for each $h\in H$; compare \cite[Subsection 1.2]{NV}.  Then we have the abelian category ${\rm Mod}^H\mbox{-}\pi_*(R)$ of $H$-graded $\pi_*(R)$-modules and its full subcategory  ${\rm mod}^H\mbox{-}\pi_*(R)$ consisting of finitely presented modules, both of which carry a natural strict $H$-action by degree-shift.

We define a functor $\pi_*\colon {\rm Mod}^G\mbox{-} R\rightarrow {\rm Mod}^H\mbox{-} \pi_*(R)$ as follows: for a $G$-graded $R$-module $X=\oplus_{g\in G}X_g$, we assign an $H$-graded $\pi_*(R)$-module $\pi_*(X)$ such that $\pi_*(X)=X$ as an ungraded $R$-module and that its homogeneous component is given by $\pi_*(X)_h=\oplus_{g\in \pi^{-1}(h)} X_g$. The functor $\pi_*$ acts on homomorphisms by the identity. We observe that $\pi_*$ sends $R$ to $\pi_*(R)$, and more generally, $\pi_*$ sends $R(g)$ to $\pi_*(R)(h)$ for each $g\in G$ and $h=\pi(g)$. We observe also that the functor $\pi_*$ is exact.

The functor $\pi_*$ has a right adjoint $\pi^*\colon {\rm Mod}^H\mbox{-} \pi_*(R)\rightarrow {\rm Mod}^G\mbox{-}R$, which is defined as follows. For an $H$-graded $\pi_*(R)$-module $Y=\oplus_{h\in H} Y_h$, we define a $G$-graded abelian group $\pi^*(Y)$ such that its  homogeneous component $\pi^*(Y)_{g}=Y_{\pi(g)}$ for each $g\in G$. A homogeneous element $r\in R_{g'}$ acts on an element $y\in \pi^*(Y)_g$ as the given action $y.r$ on $Y$, where the resulting element $y.r\in Y_{\pi(g'g)}$ is viewed now as an element in $\pi^*(Y)_{gg'}$. This defines a $G$-graded $R$-module $\pi^*(Y)$. For an $H$-graded $\pi_*(R)$-module homomorphism $f\colon Y\rightarrow Y'$, the corresponding homomorphism $\pi^*(f)\colon \pi^*(Y)\rightarrow \pi^*(Y')$ sends an element $y\in \pi^*(Y)_g=Y_{\pi(g)}$ to $f(y)\in Y'_{\pi(g)}=\pi^*(Y')_g$ for each $g\in G$. We observe that the functor $\pi^*$ is exact.

The adjoint pair $(\pi_*, \pi^*)$ is given by the following natural isomorphism
\begin{align*}
{\rm Hom}_{{\rm Mod}^H\mbox{-}\pi_*(R)} (\pi_*(X), Y)\stackrel{\sim}\longrightarrow {\rm Hom}_{{\rm Mod}^G\mbox{-}R} (X, \pi^*(Y))
\end{align*}
 which sends $f\colon \pi_*(X)\rightarrow Y$ to $f'\colon X\rightarrow \pi^*(Y)$ such that $f'_g\colon X_g\rightarrow \pi^*(Y)_g=Y_{\pi(g)}$ is the restriction of $f_{\pi(g)}\colon \pi_*(X)_{\pi(g)}\rightarrow Y_{\pi(g)}$ to the direct summand $X_g$. It follows that the unit $\eta\colon {\rm Id}_{{\rm Mod}^G\mbox{-}R}\rightarrow \pi^*\pi_*$ is given such that $(\eta_X)_g\colon X_g\rightarrow \pi^*\pi_*(X)_g=\oplus_{g'\in \pi^{-1}(\pi(g))} X_{g'}$ is the inclusion of $X_g$. The counit $\epsilon\colon \pi_*\pi^*\rightarrow {\rm Id}_{{\rm Mod}^H\mbox{-}\pi_*(R)}$ is given such that $(\epsilon_Y)_h\colon \pi_*\pi^*(Y)_h=\oplus_{g\in \pi^{-1}(h)} Y_{\pi(g)}\rightarrow Y_h$ maps each direct summand $Y_{\pi(g)}$ identically to $Y_h$ if $h\in \pi(G)$, and that $(\epsilon_Y)_h=0$ otherwise.

 \begin{lem}\label{lem:monadN}
 Let $N$ be the kernel of $\pi\colon G\rightarrow H$. Then the monad defined by the adjoint pair $(\pi_*, \pi^*; \eta, \epsilon)$ coincides with the monad defined by the degree-shift action of $N$ on ${\rm Mod}^G\mbox{-}R$.
 \end{lem}

\begin{proof}
Denote by $(M, \eta', \mu)$ the monad that is defined by the degree-shift action of $N$ on ${\rm Mod}^G\mbox{-}R$. As we computed above, $\pi^*\pi_*(X)_g=\oplus_{g'\in \pi^{-1}(\pi(g))} X_{g'}=\oplus_{n\in N} X(n^{-1})_g$ for each $g\in G$. Indeed, this proves that the endofunctor $\pi^*\pi_*$ equals $M$, which is by definition $\oplus_{n\in N} F_n=\oplus_{n\in N}(n^{-1})$. It is direct to verify that $\eta=\eta'$ and $\mu=\pi^* \epsilon \pi_*$.
\end{proof}

We consider the degree-shift action of $N$ on ${\rm Mod}^G\mbox{-}R$, and thus the category $({\rm Mod}^G\mbox{-}R)^N$ of $N$-equivariant objects.  Consider the following functor
\begin{align}\label{equ:theta}
\Theta\colon  {\rm Mod}^H\mbox{-}\pi_*(R)\longrightarrow ({\rm Mod}^G\mbox{-}R)^N
\end{align}
sending $Y$ to $\Theta(Y)=(\pi^*(Y), {\rm Id})$, where ${\rm Id}_n\colon \pi^*(Y)\rightarrow \pi^*(Y)(n^{-1})$ is the identity for each $n\in N$; here, we observe that $\pi^*(Y)(n^{-1})=\pi^*(Y)$. The functor $\Theta$  sends a homomorphism $f$ to $\pi^*(f)$.

The main observation is as follows. We mention that it is implicitly due to \cite[Theorem 6.4.1]{NV}; see Corollary \ref{cor:NV}.

\begin{prop}\label{prop:app2}
Keep the notation as above. Then the functor $\Theta$ is an equivalence if and only if $\pi\colon G\rightarrow H$ is surjective. In this case, if in addition $N$ is finite, the equivalence $\Theta$ restricts to an equivalence
$$\Theta\colon {\rm mod}^H\mbox{-}\pi_*(R)\stackrel{\sim}\longrightarrow ({\rm mod}^G\mbox{-}R)^N.$$
\end{prop}

\begin{proof}
We claim that the functor $\pi^*$ is faithful if and only if $\pi$ is surjective. Recall that $\pi^*(Y)_g=Y_{\pi(g)}$ for each $g\in G$, and that for a morphism $f\colon Y\rightarrow Y'$, $\pi^*(f)_g=f_{\pi(g)}$. In particular, $\pi^*(f)=0$ if and only if $f_{h}=0$ for all $h\in \pi(G)$. This proves the ``if" part of the claim. On the other hand, if $\pi$ is not surjective, we take $h\in H$ that is not contained in $\pi(G)$. Consider $Y=\pi_*(R)(h)$ in ${\rm Mod}^H\mbox{-}\pi_*(R)$. Then $\pi^*(Y)=0$, or equivalently, $\pi^*({\rm Id}_Y)=0$. This shows that $\pi^*$ is not faithful in this case.

Set $\mathcal{A}={\rm Mod}^G\mbox{-}R$. Denote by $M$ the monad on $\mathcal{A}$ defined by the adjoint pair $(\pi_*, \pi^*)$. We consider the associated comparison functor $K\colon {\rm Mod}^H\mbox{-}\pi_*(R)\rightarrow M\mbox{-Mod}_{\mathcal{A}}$. By Lemma \ref{lem:monadN}, Proposition \ref{prop:monadic}  and Remark \ref{rem:monadic}, we identify $ M\mbox{-Mod}_{\mathcal{A}}$ with $\mathcal{A}^N$, with which the functor $\Theta$ is identified with this comparison functor $K$. Here, one compares the definition of $\Theta$ in (\ref{equ:theta})  with the construction of $K$ in (\ref{equ:K}). Recall that the functor $\pi^*$ is exact. Then we apply Theorem \ref{thm:Beck} to the adjoint pair $(\pi_*, \pi^*)$,  and infer that $\Theta$ is an equivalence if and only if $\pi^*$ is faithful, which by the above claim is equivalent to the surjectivity of $\pi$.

The restricted equivalence follows from Lemma \ref{lem:fp}(2).
\end{proof}

The following fact is well known.

\begin{lem}\label{lem:fp}
Assume that the homomorphism $\pi\colon G\rightarrow H$ is surjective with its kernel $N$ finite. Let $X$ be a $G$-graded $R$-module and $Y$  an $H$-graded $\pi_*(R)$-module. Keep the notation as above. Then the following statements hold.
\begin{enumerate}
\item The $G$-graded $R$-module $X$ is finitely presented if and only if so is $\pi_*(X)$.
\item The $H$-graded $\pi_*(R)$-module $Y$ is finitely presented if and only if so is $\pi^*(Y)$.
\end{enumerate}
\end{lem}

\begin{proof}
Recall that $\pi_*(R(g))=\pi_*(R)(h)$ for each $g\in G$ and $h=\pi(g)$. Thus the exact functor $\pi_*$ preserves finitely generated projective modules. Then the ``only if" part of (1) follows.

We observe that $\pi^*(\pi_*(R)(h)) \simeq \oplus_{n\in N} R(ng)$ for $h=\pi(g)$, and that $N$ is finite. It follows that the exact functor $\pi^*$ preserves finitely generated projective modules. The ``only if" part of (2) follows.

For the ``if" part of (1), we assume that $\pi_*(X)$ is finitely presented, and thus so is $\pi^*\pi_*(X)$. The unit $\eta_X\colon X\rightarrow \pi^*\pi_*(X)$ is a split monomorphism. It follows that $X$ is finitely presented.

It remains to prove the ``if" part of (2). We claim that for each $H$-graded $\pi_*(R)$-module $Z$, if  $\pi^*(Z)$ is finitely generated, so is $Z$. Indeed, we recall that the exact functor $\pi_*$ preserves finitely generated projective modules. It follows that $\pi_*\pi^*(Z)$ is finitely generated. Since the counit $\epsilon_Z\colon \pi_*\pi^*(Z)\rightarrow Z$ is surjective, we infer that $Z$ is finitely generated. Here, for the surjectivity of $\epsilon_Z$ we use the surjectivity of the homomorphism $\pi\colon G\rightarrow H$.

We assume that $\pi^*(Y)$ is finitely presented. By the claim $Y$ is finitely generated. Take an exact sequence $0\rightarrow K\rightarrow P\rightarrow Y\rightarrow 0$ in ${\rm Mod}^H\mbox{-}\pi_*(R)$ with $P$ finitely generated projective. Applying the exact functor $\pi^*$ and the fact that $\pi^*(P)$ is finitely generated projective, we infer that $\pi^*(K)$ is finitely generated. By the claim again, we have that $K$ is finitely generated. This shows that $Y$ is finitely presented.
\end{proof}

\subsection{} We will relate Proposition \ref{prop:app2} to \cite[Theorem 6.4.1]{NV}.

Let $N$ be a normal subgroup of a group $G$. Let $R=\oplus_{g\in G} R_g$ be a $G$-graded ring. The  \emph{graded group ring} $R^{\rm gr}[N]$ is defined as follows:  $R^{\rm gr}[N]=\oplus_{n\in N} Ru_n$ is a free left $R$-module with a basis $\{u_{n}\;|\;  n\in N\}$, which is $G$-graded by means of $|ru_n|=|r| \cdot n$ for a homogeneous element $r$ in $R$; the multiplication is given as follows $$(ru_n) (r'u_{n'})=rr'u_{(|r'|^{-1}\cdot n\cdot |r|')n'}.$$
 We mention that the elements $u_n$ are invertible.

 We observe that $R^{\rm gr}[N]$ is a $G$-graded ring, and that the canonical map $R\rightarrow R^{\rm gr}[N]$, sending $r$ to $ru_{e}$, identifies $R$ as a $G$-graded subring of $R^{\rm gr}[N]$. We refer for the details to \cite[Subsection 6.1]{NV}.

The following result establishes the link between the graded group ring $R^{\rm gr}[N]$ and the category of $N$-equivariant objects in ${\rm Mod}^G\mbox{-}R$. We emphasize that here we consider the degree-shift action of $N$ on ${\rm Mod}^G\mbox{-}R$.

\begin{prop}\label{prop:iso}
Keep the notation as above. Then there is an isomorphism  of categories
$$({\rm Mod}^G\mbox{-}R)^N\stackrel{\sim}\longrightarrow {\rm Mod}^G\mbox{-}(R^{\rm gr}[N]).$$
\end{prop}

\begin{proof}
Recall that an $N$-equivariant object $(X, \alpha)$ is a $G$-graded $R$-module $X$ with isomorphisms $\alpha_n\colon X\rightarrow F_nX=X(n^{-1})$ in ${\rm Mod}^G\mbox{-}R$, subject to the relations $F_n(\alpha_{n'})\circ \alpha_n=\alpha_{nn'}$. We also view $\alpha_n$ as an automorphism on $X$  of degree $n^{-1}$.

We define a functor $\Delta \colon ({\rm Mod}^G\mbox{-}R)^N\stackrel{\sim}\longrightarrow {\rm Mod}^G\mbox{-}(R^{\rm gr}[N])$ as follows. To an $N$-equivariant object $(X, \alpha)$, we associate a $G$-graded $R^{\rm gr}[N]$-module $\Delta(X, \alpha)=X$ such that the $R$-action is the same as $X$ and that $$x.u_n=(\alpha_{|x|\cdot n\cdot |x|^{-1}})^{-1}(x)$$
for each homogeneous element $x$ in $X$ and $n\in N$. It is rountine to verify that this defines a $G$-graded $R^{\rm gr}[N]$-module structure on $X$. The functor $\Delta$ acts by the identity on morphisms. Roughly speaking, the functor $\Delta$ rearranges the isomorphisms $\alpha_n$'s  into the action of the invertible elements $u_n$'s.

The inverse functor of $\Delta$ associates to a $G$-graded $R^{\rm gr}[N]$-module $Y$ the $N$-equivaraint object $(Y, \beta)$, where the isomorphism $\beta_{n}\colon Y\rightarrow Y(n^{-1})$ is given by $\beta_n(y)=y.(u_{|y|^{-1}\cdot n^{-1}\cdot |y|})$ for each homogeneous element $y\in Y$.
\end{proof}

We combine Propositions \ref{prop:app2} and \ref{prop:iso} to recover the following result.

\begin{cor}\label{cor:NV} {\rm (\cite[Theorem 6.4.1]{NV})} Let $R=\oplus_{g\in G}R_g$ be a $G$-graded ring and let $N\subseteq G$ be a normal subgroup. Consider the canonical projection $\pi\colon G\rightarrow G/N$. Then there is an equivalence of categories
$${\rm Mod}^{G/N}\mbox{-}\pi_*(R)\stackrel{\sim}\longrightarrow {\rm Mod}^G\mbox{-}(R^{\rm gr}[N]).$$
\end{cor}

\subsection{} We will show that in the case that $N$ is a finite abelian group which is a direct summand of $G$, one might recover the category of $G$-graded $R$-modules from the category of $G/N$-graded $\pi_*(R)$-modules via the equivariantization with respect to a certain twisting action; see Proposition \ref{prop:recover}.

Throughout this subsection, we assume that $A$ is a finite abelian group. Let $k$ be a splitting field of $A$. In other words, the group algebra $kA$ is isomorphic to a direct product of copies of $k$; in particular, the characteristic  of the field $k$ does not divide the order of $A$.   Denote by $\widehat{A}$ the \emph{character group} of $A$, that is, the group of linear characters of $A$ over $k$.

Let $V$ be a vector space over $k$. By a \emph{linear $\widehat{A}$-action} on $V$ we mean a group homomorphism from $\widehat{A}$ to the general linear group of $V$. By an \emph{$A$-gradation} on $V$, we mean a decomposition $V=\oplus_{a\in A} V_a$ into subspaces indexed by $A$. We have the following well-known one-to-one correspondence:
\begin{align}\label{equ:corres}
\{A\mbox{-gradations on }V\} \longleftrightarrow \{\mbox{linear }\widehat{A}\mbox{-actions on }V\}.
\end{align}
The correspondence identifies an $A$-gradation $V=\oplus_{a\in A} V_a$ with the linear  $\widehat{A}$-action given by $\chi.v=\chi(a)v$ for any $\chi \in \widehat{A}$ and $v\in V_a$. Here, $\chi.v$ denotes the left action of $\chi$ on $v$.

We will consider a restriction of this correspondence. Let $H$ be an arbitrary group and let $R=\oplus_{h\in H} R_h$ be an $H$-graded algebra over $k$. By an $\widehat{A}$-\emph{action as graded automorphisms} on $R$ we mean a group homomorphism from $\widehat{A}$ to the automorphism group  of $R$ as an $H$-graded algebra.

Consider the canonical projection $\pi\colon H\times A\rightarrow H$ which sends $(h, a)$ to $h$. An $A$-\emph{refinement} of $R$ means an $(H\times A)$-graded algebra $\bar{R}$ such that $\pi_*(\bar{R})=R$, or equivalently, each homogeneous component $R_h=\oplus_{a\in A} \bar{R}_{(h, a)}$ has an $A$-gradation such that $rr'\in \bar{R}_{(hh', aa')}$ for any elements $r\in \bar{R}_{(h, a)}$ and $r'\in \bar{R}_{(h', a')}$.

The following result is well known; compare \cite[Proposition 1.3.13 and Remarks 1.3.14]{NV}.

\begin{lem}\label{lem:corres}
Let $R$ be an $H$-graded algebra. Then there is a one-to-one correspondence
\begin{align}\label{equ:corres2}
\{A\mbox{-refinements of }R\} \longleftrightarrow \{\widehat{A}\mbox{-actions  as graded automorphisms on }R\},
\end{align}
which identifies an $A$-refinement $\bar{R}$ with the $\widehat{A}$-action given by $\chi.r=\chi(a)r$ for $\chi\in \widehat{A}$ and  $r\in \bar{R}_{(h, a)}$.
\end{lem}

\begin{proof}
We apply the correspondence (\ref{equ:corres}) to the homogeneous component $R_h$ for each $h\in H$. It suffices to show that the decomposition $R=\oplus_{(h, a)\in H\times A} \bar{R}_{(h, a)}$ makes $R=\bar{R}$ an $(H\times A)$-graded algebra if and only if the corresponding $\widehat{A}$-action on $R$ is given by $H$-graded algebra automorphisms. Take the ``only if" part for example. For any two elements $r\in \bar{R}_{(h, a)}\subseteq R_{h}$ and $r\in \bar{R}_{(h', a')}\subseteq R_{h'}$, we have
$$\chi.(rr')=\chi(aa')rr'=(\chi(a)r)(\chi(a')r')=(\chi.r)(\chi.r'),$$
where the leftmost equality uses the $(H\times A)$-gradation on $R=\bar{R}$.  This proves that $\chi$ acts on $R$ by an $H$-graded algebra automorphism.
\end{proof}

Let $\sigma\colon R\rightarrow R$ be an automorphism as an $H$-graded algebra. For each $X\in \mbox{Mod}^H\mbox{-}R$, the \emph{twisted module} $X^\sigma$ is defined such that $X^\sigma=X$  as an $H$-graded space  and that it has a  new $R$-action ``$_\circ$" given by $x_\circ r=x.\sigma(r)$. This gives rise to an automorphism $(-)^\sigma\colon  \mbox{Mod}^H\mbox{-}R\rightarrow \mbox{Mod}^H\mbox{-}R$ . In this way, each subgroup $G'$ of the automorphism group of $R$ as an $H$-graded algebra has a strict action on $\mbox{Mod}^H\mbox{-}R$ by assigning to each $g\in G'$ the automorphism $F_g=(-)^{g^{-1}}$. Such an action is referred as a \emph{twisting action} of $G'$ on $\mbox{Mod}^H\mbox{-}R$.

Let $A$ be a finite abelian group. Consider an $A$-refinement $\bar{R}$ of $R$, and the corresponding $\widehat{A}$-action on $R$. Then there is a strict $\widehat{A}$-action on $\mbox{Mod}^H\mbox{-}R$ by assigning to each $\chi\in \widehat{A}$ the automorphism $F_\chi=(-)^{\chi^{-1}}$; here, we identify $\chi$ as an element in the automorphism group of $R$. We refer to this action as the \emph{twisting action} of $\widehat{A}$ on  $\mbox{Mod}^H\mbox{-}R$ that \emph{corresponds to the $A$-refinement $\bar{R}$}. We will consider the category $(\mbox{\rm Mod}^H\mbox{-}R)^{\widehat{A}}$ of $\widehat{A}$-equivariant objects in $\mbox{Mod}^H\mbox{-}R$.

By Proposition \ref{prop:app2} there is an equivalence of categories
\begin{align}\label{equ:inverseequi}
\Theta\colon  {\rm Mod}^H\mbox{-}R\stackrel{\sim}\longrightarrow ({\rm Mod}^{(H\times A)}\mbox{-}\bar{R})^A
\end{align}
where we consider the degree-shift action of $A$ on ${\rm Mod}^{(H\times A)}\mbox{-}\bar{R}$. The following is a different  equivalence, which is somehow adjoint to (\ref{equ:inverseequi}); see Remark \ref{rem:Len1}.

Recall that the canonical projection $\pi\colon H\times A\rightarrow H$  satisfies $\pi_*(\bar{R})=R$. We define the following functor
$$\Gamma \colon \mbox{\rm Mod}^{(H\times A)}\mbox{-}\bar{R} \longrightarrow (\mbox{\rm Mod}^H\mbox{-}R)^{\widehat{A}}$$
which sends an $(H\times A)$-graded $\bar{R}$-module $\bar{X}$ to $\Gamma(\bar{X})=(\pi_*(\bar{X}), \gamma)$, where each isomorphism $\gamma_\chi\colon \pi_*(\bar{X})\rightarrow F_\chi(\pi_*(\bar{X}))=\pi_*(\bar{X})^{\chi^{-1}}$ is given by $\gamma_\chi(x)=\chi^{-1}(a)x$ for $x\in \bar{X}_{(h, a)}\subseteq \pi_*(\bar{X})_h$. The action of $\Gamma$ on morphisms is given by $\pi_*$.

The following result generalizes slightly a result in \cite{Len}.

\begin{prop}\label{prop:recover}
Let $R$ be an $H$-graded algebra over $k$, and let $A$ be a finite abelian group such that $k$ is a splitting field of $A$. Consider an $A$-refinement $\bar{R}$ of $R$ and the corresponding twisting action of $\widehat{A}$ on  $\mbox{\rm Mod}^H\mbox{-}R$. Then the above functor $\Gamma$ is an isomorphism of categories. Moreover, we have a restricted isomorphism
$$ \Gamma\colon \mbox{\rm mod}^{(H\times A)}\mbox{-}\bar{R} \stackrel{\sim}\longrightarrow (\mbox{\rm mod}^H\mbox{-}R)^{\widehat{A}}.$$
\end{prop}

\begin{proof}
Let us indicate the inverse of $\Gamma$. For an $\widehat{A}$-equivariant object $(X, \alpha)$ in $\mbox{\rm Mod}^H\mbox{-}R$, we have for each $\chi\in \widehat{A}$ the isomorphism $\alpha_\chi\colon X\rightarrow F_\chi(X)=X^{\chi^{-1}}$ in $\mbox{\rm Mod}^H\mbox{-}R$; in particular, $\alpha_\chi$ induces a $k$-linear automorphism on each homogeneous component $X_h$ for each $h\in H$. In  view of (\ref{equ:rel}), we observe that these isomorphisms $(\alpha_\chi)^{-1}=\alpha_{\chi^{-1}}$ induce a linear $\widehat{A}$-action on $X_h$. By the correspondence (\ref{equ:corres}) each homogeneous component $X_h$ has an $A$-gradation $X_h=\oplus_{a\in A}\bar{X}_{(h, a)}$ such  that $\alpha_{\chi^{-1}}(x)=\chi(a)x$ for $x\in \bar{X}_{(h, a)}$. This makes $\bar{X}=X=\oplus_{(h, a)\in H\times A} \bar{X}_{(h, a)}$ an $(H\times A)$-graded $\bar{R}$-module; moreover, we have $\alpha_\chi=\gamma_\chi$ for each $\chi\in \widehat{A}$. In other words, $(X, \alpha)=\Gamma(\bar{X})$. Then the assignment $(X, \alpha)\mapsto \bar{X}$ defines the inverse functor of $\Gamma$.

The restricted isomorphism follows from Lemma \ref{lem:fp}(1).
\end{proof}

\begin{rem}\label{rem:Len1}
We mention that one might infer the equivalence (\ref{equ:inverseequi}) from Proposition \ref{prop:recover} and a general result \cite[Theorem 7.2]{El2014}. Here, we identify $A$ with the character group of $\widehat{A}$. In other words,  these two equivalences are somehow \emph{adjoint} to each other. We thank Helmut Lenzing for this remark.
\end{rem}

\section{Weighted projective lines and quotient abelian categories}

In this section, we recall from \cite{GL87, GL90} some basic facts on weighted projective lines. For a weighted projective line, we study the relationship between  the restriction subalgebras of its homogeneous coordinate algebra and certain quotient categories of the category of coherent sheaves on it. Here, we work on an arbitrary field $k$.

\subsection{}

Let $t\geq 1$ be an integer. A \emph{weight sequence} $\mathbf{p}=(p_1, p_2, \cdots, p_t)$  with $t$ weights consists of $t$ positive integers satisfying $p_i\geq 2$. We will assume that $p_1\geq p_2\geq \cdots \geq p_t$.

The \emph{string group} $L(\mathbf{p})$ associated to a weight sequence $\mathbf{p}$ is an abelian group with generators $\vec{x}_1, \vec{x}_2, \cdots, \vec{x}_t$ subject to the relations $p_1\vec{x}_1=p_2\vec{x}_2=\cdots =p_t\vec{x}_t$. This common element is denoted by $\vec{c}$, called the \emph{canonical element}. The abelian group $L(\mathbf{p})$ is of rank one, where the canonical element $\vec{c}$ has infinite order. We observe an isomorphism $L(\mathbf{p})/{\mathbb{Z}\vec{c}}\simeq \prod_{i=1}^t \mathbb{Z}/p_i\mathbb{Z}$ of abelian groups. From these facts, we infer that  each element $\vec{x}$ in $L(\mathbf{p})$ is uniquely expressed in its \emph{normal form}
\begin{align}
\vec{x}=l\vec{c}+\sum_{i=1}^t l_i \vec{x}_i
\end{align}
with $l\in \mathbb{Z}$ and $0\leq l_i< p_i$. Set $\phi(\vec{x})=l$; this gives rise to a map $\phi\colon L(\mathbf{p})\rightarrow \mathbb{Z}$. For each $1\leq i\leq t$, we define a homomorphism $\pi_i\colon L(\mathbf{p})\rightarrow \mathbb{Z}/{p_i\mathbb{Z}}$ such that $\pi_i (\vec{x}_j)=\delta_{i, j}\bar{1}$; it is surjective.

We denote by $\mathbb{N}=\{0, 1, 2\cdots\}$ the set of natural numbers.  Define the following map
\begin{align*}{\rm mult}\colon L(\mathbf{p})\longrightarrow \mathbb{N}, \quad \vec{x}\mapsto {\rm max}\{\phi(\vec{x})+1, 0\}.
 \end{align*}

We set $p={\rm lcm}(\mathbf{p})={\rm lcm}(p_1, p_2, \cdots, p_t)$ to be the least common multiple of $\mathbf{p}$. The following homomorphism of abelian groups, called the \emph{degree map}, is well defined
\begin{align*}
\delta\colon L(\mathbf{p})\longrightarrow \mathbb{Z}, \quad \vec{x}_i\mapsto \frac{p}{p_i}.\end{align*}
The degree map is surjective and its kernel coincides with the torsion subgroup of $L(\mathbf{p})$. We observe that $\delta(\vec{c})=p$.

Recall that the \emph{dualizing element} $\vec{\omega}$ in $L(\mathbf{p})$ is defined as $\vec{\omega}=(t-2)\vec{c}-\sum_{i=1}^t\vec{x}_i$; its normal form is $\vec{\omega}=-2\vec{c}+\sum_{i=1}^t(p_i-1)\vec{x}_i$. One computes that $\delta(\vec{\omega})=p((t-2)-\sum_{i=1}^t\frac{1}{p_i})$.

 The string group $L(\mathbf{p})$ is partially ordered such that $\vec{x}\leq \vec{y}$  if and only if $\vec{y}-\vec{x}\in \mathbb{N}\{\vec{x}_1, \vec{x}_2, \cdots, \vec{x}_t\}$, which is further equivalent to $\phi(\vec{y}-\vec{x})\geq 0$. Here, for a subset $S$ of an abelian group $A$, we denote by $\mathbb{N}S$ the submonoid of $A$ generated by  $S$. We denote by $L(\mathbf{p})_+$ the positive cone of $L(\mathbf{p})$, which by definition equals the submonoid $\mathbb{N}\{\vec{x}_1, \vec{x}_2, \cdots, \vec{x}_t\}$.

\subsection{} Let $k$ be an arbitrary field. A \emph{weighted projective line} $\mathbb{X}=\mathbb{X}(\mathbf{p}, \lambda)$ is given by a \emph{parameter sequence} $\lambda=(\lambda_1, \lambda_2, \cdots, \lambda_t)$, that is, a collection of pairwise distinct rational points on the projective line $\mathbb{P}^1_k$, together with a weight sequence $\mathbf{p}=(p_1, p_2, \cdots, p_t)$. We will assume that the parameter sequence $\lambda$ is \emph{normalized} such that $\lambda_1=\infty$, $\lambda_2=0$ and $\lambda_3=1$.

The \emph{homogeneous coordinate algebra} $S=S(\mathbf{p}, \lambda)$  of a weighted projective line $\mathbb{X}$ is by definition $k[X_1,  X_2, \cdots, X_t]/I$, where $I$ is the ideal generated by $X_i^{p_i}-(X_2^{p_2}-\lambda_i X_1^{p_1}), 3\leq i\leq t$. We write $x_i=X_i+I$ in $S$.

The algebra $S$ is $L(\mathbf{p})$-graded by means of ${\rm deg}\; x_i=\vec{x}_i$. Hence, $S=\oplus_{\vec{x}\in L(\mathbf{p})} S_{\vec{x}}$, where $S_{\vec{x}}$ is the homogeneous component of degree $\vec{x}$. We observe that $S_{\vec{x}}\neq 0$ if and only if $\vec{x}\geq \vec{0}$; in this case, if $\vec{x}=l\vec{c}+\sum_{i=1}^t l_i\vec{x}_i$ is in its normal form, then the subset $\{x_1^{ap_1}x_2^{bp_2}x_1^{l_1}x_2^{l_2}\cdots x_t^{l_t}\; |\; a+b=l, a, b\geq 0\}$ is a basis of $S_{\vec{x}}$. We infer that \begin{align}\label{equ:mult}
{\rm dim}_k \; S_{\vec{x}}={\rm mult}(\vec{x})\end{align}
 for any $\vec{x}\in L(\mathbf{p})$.

For an irreducible monic polynomial $f(X)\in k[X]$, we define its \emph{homogenization} $f^h(X, Y)=X^{d}f(Y/X)\in k[X, Y]$, where $d$ is the degree of $f$. The set of all homogeneous prime ideals of $S$ is given by ${\rm Spec}^{L(\mathbf{p})}(S)=\{(0), (x_i), (f^h(x_1^{p_1}, x_2^{p_2})), \mathfrak{m}=(x_1, x_2, \cdots, x_t)\; |\; f\neq X-\lambda_i, 1\leq i\leq t\}$; compare \cite[Proposition 1.3]{GL87}. Here, $\mathfrak{m}$ is the unique homogeneous maximal ideal of $S$.

The following observation is direct. Recall that for an $L(\mathbf{p})$-graded $S$-module $X$, we denote by ${\rm gsupp}(X)$ its grading support; see Subsection 5.1.

\begin{lem}\label{lem:gsupp}
Consider the graded $S$-modules $S/\mathfrak{p}(\vec{x})$ for $\mathfrak{p}\in {\rm Spec}^{L(\mathbf{p})}(S)$ and $\vec{x}\in L(\mathbf{p})$. Write $\vec{x}=l\vec{c}+\sum_{j=1}^t l_j\vec{x}_j$ in its normal form. Then the following statements hold.
\begin{enumerate}
\item Assume that $\mathfrak{p}=(0)$ or $(f^h(x_1^{p_1}, x_2^{p_2}))$. Then $\mbox{\rm gsupp}\; S/\mathfrak{p}(\vec{x})=L(\mathbf{p})_+-\vec{x}$. In particular,  $\mathbb{N}(m_0\vec{c})\subseteq \mbox{\rm gsupp}\; S/\mathfrak{p}(\vec{x})$ for sufficiently large $m_0>0$.
\item Assume that $\mathfrak{p}=(x_i)$. Then $\mbox{\rm gsupp}\; S/\mathfrak{p}(\vec{x})=\sum_{j\neq i}\mathbb{N}\vec{x}_j-\vec{x}$. In particular, $\mathbb{N}(m_0\vec{c})-l_i\vec{x}_i\subseteq \mbox{\rm gsupp}\; S/\mathfrak{p}(\vec{x})\subseteq \pi_i^{-1}(-\bar{l_i})$ for sufficiently large $m_0>0$.
    \item Assume that $\mathfrak{p}=\mathfrak{m}$. Then $\mbox{\rm gsupp}\; S/\mathfrak{p}(\vec{x})=\{-\vec{x}\}$.
\end{enumerate}
\end{lem}

\begin{proof}
We use the facts that $\mbox{\rm gsupp}(X(\vec{x}))=\mbox{\rm gsupp}(X)-\vec{x}$, and that the quotient algebra $S/\mathfrak{p}$ is a graded domain which is generated as an algebra by $x_i$'s that are not contained in $\mathfrak{p}$. It follows that  $\mbox{\rm gsupp}(S/\mathfrak{p})=L(\mathbf{p})_+$ for $\mathfrak{p}=(0)$ or $(f^h(x_1^{p_1}, x_2^{p_2}))$, and that $\mbox{\rm gsupp}( S/(x_i))=\sum_{j\neq i} \mathbb{N}\vec{x}_j$. The remaining statements are evident.
\end{proof}

Let $H\subseteq L(\mathbf{p})$ be an infinite subgroup. Consider the \emph{restriction subalgebra} $S_H=\oplus_{\vec{x}\in H}S_{\vec{x}}$ of $S$, which will be viewed as an $H$-graded algebra. For example, $S_{\mathbb{Z}\vec{c}}$ is isomorphic to the polynomial algebra $k[X, Y]$ by identifying $X$ with $x_1^{p_1}$ and $Y$ with $x_2^{p_2}$. More generally, for $m\geq 1$ the restriction subalgebra $S_{\mathbb{Z}(m\vec{c})}$ is isomorphic to the subalgebra of $k[X, Y]$ generated by monomials of degree $m$.

\begin{lem}
Let $H\subseteq L(\mathbf{p})$ be an infinite subgroup. Then the restriction subalgebra $S_H$ is a finitely generated $k$-algebra, and as an $S_H$-module, $S$ is finitely generated.
\end{lem}
\begin{proof}
Assume that $H\cap \mathbb{Z}\vec{c}=\mathbb{Z}(m\vec{c})$ for some $m\geq 1$.  Recall from \cite[Proposition 1.3]{GL87} that $S$ is a finitely generated $S_{\mathbb{Z}\vec{c}}$-module, and thus a finitely generated $S_{\mathbb{Z}(m\vec{c})}$-module. Since the algebra $S_{\mathbb{Z}(m\vec{c})}$ is noetherian, the algebra $S_H$, viewed as an $S_{\mathbb{Z}(m\vec{c})}$-submodule of $S$, is finitely generated. Then the statements follow immediately.
\end{proof}

Recall that ${\rm mod}^{L(\mathbf{p})}\mbox{-}S$ denotes the abelian category of finitely presented $L(\mathbf{p})$-graded $S$-modules. Consider the \emph{restriction functor} $${\rm res}\colon {\rm mod}^{L(\mathbf{p})}\mbox{-}S\rightarrow {\rm mod}^H\mbox{-}S_H,$$ which sends an $L(\mathbf{p})$-graded $S$-module $X$ to $X_H=\oplus_{h\in H} X_h$, which is naturally an $H$-graded $S_H$-module. The functor ``res" is well defined by the following fact.

\begin{lem}
Let $H\subseteq L(\mathbf{p})$ be an infinite subgroup, and let $X$ be a finitely generated $L(\mathbf{p})$-graded $S$-module. Then the  $H$-graded $S_H$-module $X_H$ is finitely generated.
\end{lem}

\begin{proof}
Assume that $H\cap \mathbb{Z}\vec{c}=\mathbb{Z}(m\vec{c})$ for some $m\geq 1$. We observe that $X$, viewed as an $S_{\mathbb{Z}(m\vec{c})}$-module, is finitely generated. Thus as its submodule, $X_H$ is a finitely generated $S_{\mathbb{Z}(m\vec{c})}$-module. In particular, $X_H$ is a finitely generated $S_H$-module.
\end{proof}

The following fact is well known.

\begin{lem}\label{lem:res}
Let $H\subseteq L(\mathbf{p})$ be an infinite subgroup. Then the restriction functor  ${\rm res}\colon {\rm mod}^{L(\mathbf{p})}\mbox{-}S\rightarrow {\rm mod}^H\mbox{-}S_H$ is a quotient functor.
\end{lem}

\begin{proof}
Indeed, the functor ``res" admits a left adjoint $S\otimes_{S_H}-\colon {\rm mod}^H\mbox{-}S_H\rightarrow {\rm mod}^{L(\mathbf{p})}\mbox{-}S$. Moreover, the counit ${\rm res}\circ S\otimes_{S_H}-\longrightarrow {\rm Id}_{{\rm mod}^H\mbox{-}S_H}$ is an isomorphism. It follows from \cite[1.3 Proposition (iii)]{GZ} that the restriction functor is a quotient functor.
\end{proof}

\subsection{}

Let $S=S(\mathbf{p}, \lambda)$ be the homogenous coordinate algebra of a weighted projective line $\mathbb{X}=\mathbb{X}(\mathbf{p}, \lambda)$. Set $L=L(\mathbf{p})$. Denote by ${\rm mod}_0^{L}\mbox{-}S$ the full subcategory of ${\rm mod}^{L}\mbox{-}S$ consisting of finite dimensional modules; it is a Serre subcategory. The corresponding quotient category is denoted by ${\rm qmod}^L\mbox{-}S={\rm mod}^{L}\mbox{-}S/{{\rm mod}_0^{L}\mbox{-}S}$, and the quotient functor is denoted by $q\colon {\rm mod}^{L}\mbox{-}S\rightarrow {\rm qmod}^L\mbox{-}S$.

For each $\vec{x}\in L$, the degree-shift functor $(\vec{x})$ on ${\rm mod}^{L}\mbox{-}S$ induces an automorphism on the quotient category ${\rm qmod}^L\mbox{-}S$, which is also denoted by $(\vec{x})$. Consequently, each subgroup of $L$ has a strict action on ${\rm qmod}^L\mbox{-}S$. Set $\mathcal{S}_i=q(S/(x_i))$ for $1\leq i\leq t$.

We mention that the quotient category ${\rm qmod}^L\mbox{-}S$ is equivalent to the category ${\rm coh}\mbox{-}\mathbb{X}$ of coherent sheaves on $\mathbb{X}$, and the object $\mathcal{S}_i$ corresponds to a simple sheaf concentrated at the homogeneous prime ideal $(x_i)$; see \cite[Subsections 1.8 and 1.7]{GL87}.

Let $H\subseteq L$ be an infinite subgroup. Consider the $H$-graded restriction subalgebra $S_H$. We have the quotient category ${\rm qmod}^H\mbox{-}S_H={\rm mod}^H\mbox{-}S_H/{{\rm mod}_0^H\mbox{-}S_H}$. The restriction functor ${\rm res}\colon {\rm mod}^{L}\mbox{-}S\rightarrow {\rm mod}^{H}\mbox{-}S_H$ induces the corresponding functor betwen the quotient categories  $${\rm res}\colon {\rm qmod}^L\mbox{-}S\rightarrow {\rm qmod}^H\mbox{-}S_H,$$ which will also be referred as the \emph{restriction functor}.

We will consider the following Serre subcategory of ${\rm qmod}^L\mbox{-}S$
$$\mathcal{N}_H=\langle \mathcal{S}_i(l\vec{x}_i)\; |\; 1 \leq i\leq t, 0\leq l\leq p_i-1 \mbox{ with }  -\bar{l}\notin   \pi_i(H)\rangle.$$
That is, $\mathcal{N}_H$ is the Serre subcategory generated by these simple sheaves $ \mathcal{S}_i(l\vec{x}_i)$.

Recall that  for each $1\leq i\leq t$, the surjective homomorphism $\pi_i\colon L\rightarrow \mathbb{Z}/{p_i\mathbb{Z}}$ is defined by  $\pi_i(\vec{x}_j)=\delta_{i, j}\bar{1}$. The following concept is related to the general notion of a wide subcategory of the category formed by all the line bundles on a weighted projective line in \cite{KLM}.

 \begin{defn}\label{defn:1}
 Let $\mathbf{p}$ be a weight sequence and $L(\mathbf{p})$ be the string group.  An infinite subgroup $H\subseteq L(\mathbf{p})$ is said to be \emph{effective} provided that for each $1\leq i\leq t$,  $\pi_i(H)=\mathbb{Z}/{p_i\mathbb{Z}}$, or equivalently, $\bar{1}$ lies in $\pi_i(H)$.
 \end{defn}

 For example, any infinite subgroup of $L(\mathbf{p})$ containing the dualizing element $\vec{\omega}$ is effective, while the subgroup $\mathbb{Z}\vec{c}$ generated by $\vec{c}$ is not effective for $t\geq 1$.

The following result justifies the terminology ``effective subgroup". We mention that in spirit it is close to \cite[Proposition 8.5]{GL90}; compare the argument in \cite[Example 5.8]{GL87}. The equivalence part of the following proposition is contained in a more general result in \cite{KLM}.

\begin{prop}\label{prop:eff}
Let $L=L(\mathbf{p})$. Keep the notation as above. Then the restriction functor ${\rm res}\colon {\rm qmod}^L\mbox{-}S\rightarrow {\rm qmod}^H\mbox{-} S_H$ induces an equivalence of categories
$${\rm qmod}^L\mbox{-}S/\mathcal{N}_H\stackrel{\sim}\longrightarrow {\rm qmod}^H\mbox{-}S_H. $$
In particular, the restriction functor ${\rm res}\colon {\rm qmod}^L\mbox{-}S\rightarrow {\rm qmod}^H\mbox{-}S_H$ is an equivalence if and only if the subgroup $H\subseteq L$ is effective.
\end{prop}

\begin{proof}
The restriction functor ${\rm res}\colon {\rm mod}^L\mbox{-}S\rightarrow {\rm mod}^H\mbox{-}S_H$ is a quotient functor by Lemma \ref{lem:res}. Then by Lemma \ref{lem:quoinduced} the induced functor ${\rm res}\colon {\rm qmod}^L\mbox{-} S\rightarrow {\rm qmod}^H\mbox{-}S_H$ is also a quotient functor. It suffices to calculate its essential kernel. Instead, we will compute the essential kernel of the composite $F\colon {\rm mod}^L\mbox{-}S\stackrel{q}\rightarrow {\rm qmod}^L\mbox{-} S \stackrel{\rm res}\rightarrow {\rm qmod}^H\mbox{-} S_H$. Indeed, for an $L$-graded $S$-module $X$, $F(X)=0$ if and only if $X_H$ is a finite dimensional $S_H$-module, which is equivalent to the condition that $\mbox{\rm gsupp}(X)\cap H$ is a finite set.

We observe that $\mathbb{Z}(m\vec{c})\subseteq H$ for a sufficiently large $m>0$. Let $\vec{x}=l\vec{c}+\sum_{i=1}^tl_i\vec{x}_i$ an arbitrary element in $L$ that is written in its normal form.  It follows from Lemma \ref{lem:gsupp} that $\mbox{\rm gsupp}(S/\mathfrak{p}(\vec{x})) \cap H $ is finite if and only if $\mathfrak{p}=\mathfrak{m}$, or $\mathfrak{p}=(x_i)$ and $-\bar{l_i}\notin \pi_i(H)$.

Recall that a finitely generated $L$-graded $S$-module $X$ has a finite filtration with factors isomorphic to $S/\mathfrak{p}(\vec{x})$ for some $\mathfrak{p}\in {\rm Spec}^L(S)$ and $\vec{x}\in L$. Recall that ${\rm Ker}\; F$ is a Serre subcategory of ${\rm mod}^L\mbox{-}S$. It follows that
$${\rm Ker}\; F=\langle S/\mathfrak{m}(\vec{y}), S/(x_i)(\vec{x}) \; |\; \vec{y}\in L, 1\leq i\leq t, \vec{x} \mbox{ satisfying that } -\bar{l_i}\notin \pi_i(H)\rangle.$$
Then we are done by using the fact that the essential kernel of ``res" equals $q({\rm Ker}\; F)$, which equals $\mathcal{N}_H$. Here, we recall from \cite[Subsection 1.6]{GL87} the well-known fact that $\mathcal{S}_i(\vec{x})=\mathcal{S}_i(l_i\vec{x}_i)$ for each $1\leq i\leq t$.
\end{proof}

\section{Weighted projective lines of tubular type}

In this section, we study weighted projective lines of tubular type. They are  explicitly related to elliptic plane curves via  two different equivariantizations. We emphasize that the idea goes back to \cite[Example 5.8]{GL87} and more explicitly to \cite{Len, Lentalk}; compare \cite{Po,Hille,Ploog}.

\subsection{}

We recall that a weight sequence $\mathbf{p}=(p_1, p_2, \cdots, p_t)$ is of \emph{tubular type} provided that the dualizing element $\vec{\omega}$ in $L(\mathbf{p})$ is torsion, or equivalently, $\delta(\vec{\omega})=0$.  An elementary calculation yields a complete list of weight sequences of tubular type: $(2,2,2,2)$, $(3,3,3)$, $(4,4,2)$ and $(6,3,2)$. We observe that in each case the order of $\vec{\omega}$ equals $p={\rm lcm}(\mathbf{p})$, which further equals $p_1$.

We have the following easy observation.

\begin{lem}\label{lem:mult}
Let $\mathbf{p}$ be a weight sequence of tubular type and let $n\geq 1$. Recall that $p={\rm lcm}(\mathbf{p})$. Then the following equation holds
$$\sum_{j=0}^{p-1} {\rm mult}(n\vec{x}_1+j\vec{\omega})=n.$$
\end{lem}

\begin{proof}
We observe that ${\rm mult}(\vec{x}+\vec{c})={\rm mult}(\vec{x})+1$ provided that $\phi(\vec{x})\geq -1$. For a weight sequence $\mathbf{p}$ of tubular type, we recall that $p=p_1$ and thus $p\vec{x}_1=\vec{c}$. We observe that $\phi(n\vec{x}_1+j\vec{\omega})\geq -1$ for each $n\geq 1$ and $0\leq j< p-1$. Denote the left hand side of the equation by $f(n)$. It follows that $f(n+p)=f(n)+p$ for each $n\geq 1$.

Therefore, to show the required equation, it suffices to prove that $f(i)=i$ for $1\leq i\leq p$. These $p$ equations are easy to verify for each of the four cases. Take the case $\mathbf{p}=(6,3,2)$ for example. We have that $f(4)=\sum_{j=0}^{5} {\rm mult}(4\vec{x}_1+j\vec{\omega})=1+0+1+1+1+0=4$ and that $f(5)=\sum_{j=0}^{5} {\rm mult}(5\vec{x}_1+j\vec{\omega})=1+0+1+1+1+1=5$. We omit the details.
\end{proof}

The following consideration is inspired by \cite[Example 5.8]{GL87}. Let $\mathbf{p}$ be a weight sequence of tubular type. We consider the subgroup $H(\mathbf{p})=\mathbb{Z}(3\vec{x}_1)\oplus \mathbb{Z}\vec{\omega}$ of $L(\mathbf{p})$ generated by $3\vec{x}_1$ and $\vec{\omega}$. It inherits the partial order from $L(\mathbf{p})$ such that its positive cone equals $H(\mathbf{p})_+=H(\mathbf{p})\cap L(\mathbf{p})_+$, which is a submonoid of $H(\mathbf{p})$. We consider the subset $3\vec{x}_1+\mathbb{Z}\vec{\omega}=\{3\vec{x}_1+j\vec{\omega}\; |\; j=0, 1,\cdots, p-1\}$ of $H(\mathbf{p})$.  We list in Table 1 all the positive elements in $3\vec{x}_1+\mathbb{Z}\vec{\omega}$ explicitly.

\begin{table}[h]
\caption{The list of generators of $H(\mathbf{p})_+$}
\begin{tabular}{|c|c|c|c|}
  \hline
  the weight sequence $\mathbf{p}$ & \multicolumn{2}{c}{ \quad  the elements in $(3\vec{x}_1+\mathbb{Z}\vec{\omega})\cap H(\mathbf{p})_+$} &  \\
  \hline
  $(2,2,2,2)$ &   \;  $3\vec{x}_1$ \;  & \multicolumn{1}{c}{\; $3\vec{x}_1+\vec{\omega}=\vec{x}_2+\vec{x}_3+\vec{x}_4$}  &\\
  \hline
    $(3,3,3)$ &   \; $3\vec{x}_1$ \;  &  \multicolumn{1}{c}{\; $3\vec{x}_1+2\vec{\omega}=\vec{x}_1+\vec{x}_2+\vec{x}_3$} &   \\
  \hline
    $(4,4,2)$  & \;  $3\vec{x}_1$ \; &   $3\vec{x}_1+2\vec{\omega}=\vec{x}_1+2\vec{x}_2$ & $3\vec{x}_1+3\vec{\omega}=\vec{x}_2+\vec{x}_3$  \\
  \hline
    $(6,3,2)$  & \;  $3\vec{x}_1$ \;  &  $3\vec{x}_1+2\vec{\omega}=\vec{x}_1+\vec{x}_2$&  $3\vec{x}_1+3\vec{\omega}=\vec{x}_3$ \\
  \hline
\end{tabular}
\end{table}

We have  the following elementary observation, for which we include an elementary proof here. For a conceptual reasoning, we refer to Remark \ref{rem:thm}.

\begin{lem}\label{lem:gene}
Let $\mathbf{p}$ be a weight sequence of tubular type. Then we have $$H(\mathbf{p})_+=\mathbb{N}((3\vec{x}_1+\mathbb{Z}\vec{\omega})\cap H(\mathbf{p})_+).$$
In other words, the monoid $H(\mathbf{p})_+$ is generated by the subset $(3\vec{x}_1+\mathbb{Z}\vec{\omega})\cap H(\mathbf{p})_+$.
\end{lem}

\begin{proof}
Recall that $p={\rm lcm}(\mathbf{p})=p_1$ equals the order of $\vec{\omega}$. For each $0\leq j\leq p-1$, we define $m(j)$ to be the least natural number $m$ such that $m(3\vec{x}_1)+j\vec{\omega}\geq \vec{0}$. We observe that $m(0)=0$ and $m(j)\geq 1$ for $j\neq 0$.

 We claim that $m(j)(3\vec{x}_1)+j\vec{\omega}$ lies in $\mathbb{N}((3\vec{x}_1+\mathbb{Z}\vec{\omega})\cap H(\mathbf{p})_+)$ for each $0\leq j\leq p-1$. Then we are done. Indeed, for any  positive element $\vec{x}=m(3\vec{x}_1)+j\vec{\omega}$, we have $m\geq m(j)$, and thus $\vec{x}=(m-m(j))(3\vec{x}_1)+(m(j)(3\vec{x}_1)+j\vec{\omega})$, which lies in $\mathbb{N}((3\vec{x}_1+\mathbb{Z}\vec{\omega})\cap H(\mathbf{p})_+)$.

 The claim is easily verified case by case.
\end{proof}

\subsection{}

We recall that a weighted projective line $\mathbb{X}=\mathbb{X}(\mathbf{p}, \lambda)$ over a field $k$ is of \emph{tubular type} provided that the weight sequence $\mathbf{p}$ is of tubular type. Recall that $S=S(\mathbf{p}, \lambda)$ is the homogeneous coordinate algebra of $\mathbb{X}$.

 We fix the notation for $S$ in the four tubular types. For the weight type  $(2,2,2,2)$, there exists a parameter $\lambda \neq 0, 1$ in the presentation of the corresponding homogeneous coordinate algebra $S$. In other three cases, we omit the normalized parameter sequences. Hence, we refer to $(2,2,2,2; \lambda)$, $(3,3,3)$, $(4,4,2)$ and $(6,3,2)$ as the \emph{type} of $S$. We emphasize that $\lambda\in k$, which is not equal to $0$ or $1$.

 We list the homogeneous coordinate algebras $S$ explicitly, according to their types. We define the \emph{associated} $\mathbb{Z}$-graded algebras $R$ of the same type on the right hand side, where all the variables are of degree one.

{\tiny
 \begin{align*}
&S(2,2,2,2; \lambda)=\frac{k[X_1, X_2, X_3, X_4]}{(X_3^2-(X_2^2-X_1^2), X_4^2-(X_2^2-\lambda X_1^2))}, \quad R(2,2,2,2;\lambda)= \frac{k[X,Y, Z]}{(Y^2Z-X(X-Z)(X-\lambda Z))};
 \\ &S(3,3,3)=\frac{k[X_1, X_2, X_3]}{(X_3^3-(X_2^3-X_1^3))},  \hskip 98pt  R(3,3,3)=\frac{k[X, Y, Z]}{(X^3-(Y^2Z-Z^2Y))};\\
&S(4,4,2)=\frac{k[X_1, X_2, X_3]}{(X_3^2-(X_2^4-X_1^4))},  \hskip 98pt  R(4,4,2)=\frac{k[X, Y, Z]}{(Y^2Z-X(X-Z)(X+Z))};\\
 &S(6,3,2)=\frac{k[X_1, X_2, X_3]}{(X_3^2-(X_2^3-X_1^6))},  \hskip 98pt  R(6,3,2)=\frac{k[X, Y, Z]}{(Y^2Z-(X^3-Z^3))}.
\end{align*}
}

We denote by $\mathbb{E}$ the projective plane curve defined by $R$. By Serre's theorem, there is an equivalence between the category ${\rm coh}\mbox{-}\mathbb{E}$ of  coherent sheaves on $\mathbb{E}$ and the quotient category ${\rm qmod}^\mathbb{Z}\mbox{-}R$ of ${\rm mod}^\mathbb{Z}\mbox{-}R$ by the Serre subcategory ${\rm mod}_0^\mathbb{Z}\mbox{-}R$ consisting of finite dimensional modules.

\begin{rem}\label{rem:j}
We mention that $\mathbb{E}(2,2,2,2;\lambda)$ is an elliptic curve for $\lambda \neq 0, 1$.  For other types, we have requirement on the characteristic of the field $k$;  see also \cite{Len}. If ${\rm char}\; k\neq 3$, $\mathbb{E}(3,3,3)$ is an elliptic curve with $j$-invariant $0$; if ${\rm char}\; k\neq 2$, $\mathbb{E}(4,4,2)$ is an elliptic curve with $j$-invariant $1728$; if ${\rm char}\; k\neq 2, 3$, $\mathbb{E}(6,3,2)$ is an elliptic curve with $j$-invariant $0$.
\end{rem}

 We define three homogeneous elements $x$, $y$ and $z$ in $S$ according to Table 2.  By comparing with Table 1, all their degrees lie in $(3\vec{x}_1+\mathbb{Z}\vec{\omega})\cap H(\mathbf{p})_+$. For an explanation of the term ``extra degree"  in Table 2, we refer to Remark \ref{rem:extra}.

\begin{table}[h]
\caption{The elements $x$, $y$ and  $z$ in $S$ and their extra degrees}
\begin{tabular}{|c|c|c|c|c|c|c|c|}
  \hline
  the types of  $S$ and $R$ &  $x$  &  $y$  & $z$  & ${\rm deg}\; x$ & ${\rm deg}\; y$ & ${\rm deg}\; z$ \\
  \hline
  $(2,2,2,2; \lambda)$ &   $x_1x_2^2$  & $x_2x_3x_4$  & \;  $x_1^3$ \; & {\small $ (1, \bar{0})$} &{\small $(1, \bar{1})$} & {\small $(1, \bar{0})$} \\
  \hline
    $(3,3,3)$ &  $x_1x_2x_3$ & $x_2^3$ & \;  $x_1^3$ \; & {\small $(1, \bar{2})$} & {\small $(1, \bar{0})$} & {\small $(1, \bar{0})$}   \\
  \hline
    $(4,4,2)$  & $x_1x_2^2$ &   $x_2x_3$ & \;  $x_1^3$ \; & {\small $(1, \bar{2})$} & {\small $(1, \bar{3})$} & {\small $(1, \bar{0})$} \\
  \hline
    $(6,3,2)$  & $x_1x_2$ & $x_3$ & \;  $x_1^3$ \;  & {\small $(1, \bar{2})$} & {\small $(1, \bar{3})$} & {\small $(1, \bar{0})$} \\
  \hline
\end{tabular}
\end{table}

 We consider the subalgebra $k[x, y, z]$ of $S$ generated by $x$, $y$ and $z$. We observe that $k[x, y, z]\subseteq S_{H(\mathbf{p})}$. Consider the surjective homomorphism $\pi\colon H(\mathbf{p})\rightarrow \mathbb{Z}$ defined by $\pi(3\vec{x}_1)=1$ and $\pi(\vec{\omega})=0$.

 The central technical result of this section is as follows.

 \begin{thm}\label{thm:R-S}
 Let $S$  be the homogeneous coordinate algebra of a weighted projective line $\mathbb{X}$ of tubular type and let $R$ be the associated $\mathbb{Z}$-graded algebra of the same type. Let $H=H(\mathbf{p})$. Keep the above notation. Then we have  $k[x, y, z]=S_H$, and that there is an isomorphism of $\mathbb{Z}$-graded algebras
 $$R \stackrel{\sim}\longrightarrow  \pi_*(S_H).$$
 \end{thm}

\begin{proof}
We claim that $k[x, y, z]\subseteq S$ is an integral extension. Recall that an integral extension preserves Krull dimensions and that the Krull dimension of the algebra $S$ is two. Hence the Krull dimension of $k[x, y, z]$ is also two.

 Recall that $S_{\mathbb{Z}\vec{c}}=k[x_1^{p_1}, x_2^{p_2}]$ and that $S$ is a finitely generated $S_{\mathbb{Z}\vec{c}}$-module. Hence $S$ is integral over $S_{\mathbb{Z}\vec{c}}$ and thus over the subalgebra $k[x_1, x_2]$. For the claim, it suffices to show that both $x_1$ and $x_2$ are integral over $k[x, y, z]$. Since $x_1^3=z$, the element $x_1$ is integral over $k[x, y, z]$.

 We consider the element $x_2$ case by case. We take  $\mathbf{p}=(2,2,2,2)$ for example, while other cases are similar and even easier. Observe in Table 2 that  $y^2=x_2^2(x_2^2-x_1^2)(x_2^2-\lambda x_1^2)$. This implies that $x_2$ is integral over $k[x, y, z, x_1]$ and thus over $k[x, y,z]$. We are done with the claim.

 We define a homomorphism $\theta\colon R\rightarrow S$ of algebras by sending $X$ to $x$, $Y$ to $y$ and $Z$ to $z$; see Table 2. The relation of $R$ is satisfied by an elementary calculation. The image of $\theta$ is $k[x, y, z]$. Recall that $R$ is an integral domain of Krull dimension two. Comparing the Krull dimensions of $R$ and the image of $\theta$, we infer that $\theta$ is injective. Consequently, we have an injective homomorphism of $\mathbb{Z}$-graded algebras $\theta \colon R\rightarrow \pi_*(S_H)$. Here, we recall that $k[x, y,z]\subseteq S_H$.

 We claim that the injective homomorphism $\theta \colon R\rightarrow \pi_*(S_{H})$ is an isomorphism. Then the two  required statements follow. Indeed, it suffices to show that for each $n\geq 1$, ${\rm dim}_k\; R_n={\rm dim}_k\; \pi_*(S_H)_n$. Recall that $\pi_*(S_H)_n=\oplus_{j=0}^{p-1} S_{3n\vec{x}_1+j\vec{\omega}}$, and thus has dimension $\sum_{j=0}^{p-1} {\rm mult}(3n\vec{x}_1+j\vec{\omega})$; see (\ref{equ:mult}). By Lemma \ref{lem:mult} we have ${\rm dim}_k\; \pi_*(S_H)_n=3n$. On the other hand, it is well known that ${\rm dim}_k\; R_n=3n$. Then we are done.
\end{proof}

\begin{rem}\label{rem:thm}
Recall that $\vec{x}\in H(\mathbf{p})$ is positive if and only if $(S_{H(\mathbf{p})})_{\vec{x}} =S_{\vec{x}} \neq 0$. We deduce from the equality $k[x, y, z]=S_{H(\mathbf{p})}$ that each $\vec{x}\in H(\mathbf{p})_+$ lies in the submonoid generated by the degrees of $x$, $y$ and $z$. This gives another proof of Lemma \ref{lem:gene}.
\end{rem}

\begin{rem} \label{rem:extra}
We now explain the ``extra degrees" of $x$, $y$ and $z$ in $S_{H(\mathbf{p})}$. We recall that $p={\rm lcm}(\mathbf{p})$ equals the order of $\vec{\omega}$. So we have the following isomorphism of abelian groups
\begin{align}
\psi\colon H(\mathbf{p})=\mathbb{Z}(3\vec{x}_1)\oplus \mathbb{Z}\vec{\omega}\stackrel{\sim}\longrightarrow \mathbb{Z}\times \mathbb{Z}/{p\mathbb{Z}}
\end{align}
such that $\psi(3\vec{x}_1)=(1, \bar{0})$ and $\psi(\vec{\omega})=(0, \bar{1})$. By this isomorphism, the $H(\mathbf{p})$-graded algebra  $S_{H(\mathbf{p})}$ is  $\mathbb{Z}\times \mathbb{Z}/p\mathbb{Z}$-graded. This yields the \emph{extra degrees} of $x$, $y$ and $z$, which are computed by applying $\psi$ to the degrees of $x$, $y$ and $z$; compare Table 1.

We observe that the algebras $R$ are naturally $\mathbb{Z}\times \mathbb{Z}/{p\mathbb{Z}}$-graded according to the degrees in Table 2; here, we abuse $X$, $Y$, $Z$ with $x$, $y$ and $z$, respectively. We denote the resulting $\mathbb{Z}\times \mathbb{Z}/{p\mathbb{Z}}$-graded algebras by $\bar{R}$; it is a $\mathbb{Z}/{p\mathbb{Z}}$-refinement of $R$.

Then the isomorphism $\theta$ in the above proof, which is given according to Table 2, yields  an isomorphism of $\mathbb{Z}\times \mathbb{Z}/{p\mathbb{Z}}$-algebras
\begin{align}\label{equ:iso1}
\theta\colon \bar{R}\stackrel{\sim}\longrightarrow S_{H(\mathbf{p})}.
\end{align}
\end{rem}

\vskip 10pt

Let $\mathbf{p}$ be a weight sequence of tubular type. Recall that $p={\rm lcm}(\mathbf{p})$. Denote by $C_p=\langle g\; |\; g^p=1\rangle $ the cyclic group of order $p$.

We assume that $k$ is a splitting field of $C_p$.  Equivalently, $p$ is invertible in $k$ which contains a primitive $p$-th root of unity. Fix $\omega\in k$ a primitive $p$-th root of unity. For example, if $\mathbf{p}=(2,2,2,2)$, then the characteristic of $k$ is not equal to $2$ and $\omega=-1$.

We identify $C_p$ with the character group $\widehat{\mathbb{Z}/{p\mathbb{Z}}}$ of $\mathbb{Z}/{p\mathbb{Z}}$  such that $g(\bar{1})=\omega$. Applying the correspondence (\ref{equ:corres2}) to the above $\mathbb{Z}/{p\mathbb{Z}}$-refinement $\bar{R}$ of the $\mathbb{Z}$-graded algebra $R$, we obtain in Table 3 the corresponding $C_p$-action on $R$ as $\mathbb{Z}$-graded algebra automorphisms.

\begin{table}[h]
\caption{The cyclic group action on $R$}
\begin{tabular}{|c|c|c|}
  \hline
  the types of  $R$ & the action of $C_p=\langle g\; |\; g^p=1\rangle $ on generators & the order of $\omega\in k$ \\
  \hline
  $(2,2,2,2; \lambda)$ &   $g.x=x, \; g.y=\omega y,\;  g.z=z$ & 2\\
  \hline
    $(3,3,3)$ &  $g.x=\omega^2 x, \; g.y=y, \; g.z=z$ & 3\\
  \hline
    $(4,4,2)$  & $g.x=\omega^2 x, \;g.y=\omega^3 y, \; g.z=z$ & 4\\
    \hline
    $(6,3,2)$  &  $g.x=\omega^2 x,\;  g.y=\omega^3 y, \; g.z=z$ &  6\\
    \hline
\end{tabular}
\end{table}

We mention that the  cyclic  group action on $R$ corresponds to a ramified Galois covering from the elliptic curve $\mathbb{E}$ to the projective line $\mathbb{P}_{k}^1$, where the Galois group is isomorphic to the acting cyclic group; compare \cite[Subsection 1.3]{Po}.

\subsection{} We will formulate the main result of this paper. As mentioned earlier, the result is implicitly contained in \cite[Example 5.8]{GL87} and later explicitly in \cite{Len} with more details; see also \cite{Lentalk}. The treatment here is slightly different from \cite{Len, Lentalk}, and we strengthen slightly the results therein; compare Remark \ref{rem:main}(3). We mention that related results appear in \cite{Hille}. 

Let $k$ be an arbitrary field. Let $S$ be the homogeneous coordinate algebra of a weighted projective line $\mathbb{X}$ of tubular type, and $R$ be the associated $\mathbb{Z}$-graded algebra of the same type. Let $\mathbf{p}$ be the weight sequence and denote $p={\rm lcm}(\mathbf{p})$. Recall that the dualizing element $\vec{\omega}$ in $L=L(\mathbf{p})$ has order $p$. The subgroup $\mathbb{Z}\vec{\omega}$ acts on ${\rm mod}^L\mbox{-}S$ by the degree-shift action, which induces the \emph{degree-shift action} of $\mathbb{Z}\vec{\omega}$ on ${\rm qmod}^L\mbox{-}S$. We consider the category $({\rm qmod}^L\mbox{-}S)^{\mathbb{Z}\vec{\omega}}$ of equivariant objects.

Assume that $k$ is a splitting field of the cyclic group $C_p$ of order $p$. Consider the $C_p$-action on $R$ as $\mathbb{Z}$-graded algebra automorphisms in Table 3. Such an action induces the twisting action of $C_p$ on ${\rm mod}^\mathbb{Z}\mbox{-}R$, which further induces the \emph{twisting action} of $C_p$ on ${\rm qmod}^\mathbb{Z}\mbox{-}R$. We consider the category $({\rm qmod}^\mathbb{Z}\mbox{-}R)^{C_p}$ of equivariant objects.

\begin{thm}{}\rm (Lenzing-Meltzer)\label{thm:main}
Let $k$ be a field. Let $S$ be the homogeneous coordinate algebra of a weighted projective line $\mathbb{X}$ of tubular type, and let $R$ be the associated $\mathbb{Z}$-graded algebra of the same type, where we keep the notation as above. Then the following statements hold.
\begin{enumerate}
\item There is an equivalence of categories
$$({\rm qmod}^L\mbox{-}S)^{\mathbb{Z}\vec{\omega}} \stackrel{\sim}\longrightarrow {\rm qmod}^\mathbb{Z}\mbox{-}R.$$
\item Assume that $k$ is a splitting field of $C_p$. Then there is an equivalence of categories
$${\rm qmod}^L\mbox{-}S\stackrel{\sim}\longrightarrow ({\rm qmod}^\mathbb{Z}\mbox{-}R)^{C_p}.$$
\end{enumerate}
\end{thm}

\begin{proof}
Recall the subgroup  $H=H(\mathbf{p})=\mathbb{Z}(3\vec{x}_1)\oplus \mathbb{Z}\vec{\omega}$ of $L$ from Subsection 7.1; it is an effective subgroup in the sense of Definition \ref{defn:1}. By Proposition \ref{prop:eff} the restriction functor is an equivalence $${\rm res}\colon {\rm qmod}^L\mbox{-}S \stackrel{\sim}\longrightarrow {\rm qmod}^H\mbox{-}S_H.$$

Recall from Remark \ref{rem:extra} the $\mathbb{Z}/{p\mathbb{Z}}$-refinement $\bar{R}$ of the $\mathbb{Z}$-graded algebra $R$. By the isomorphism $\theta$ in (\ref{equ:iso1}) we have an isomorphism of categories $$\theta^*\colon {\rm qmod}^H\mbox{-}S_H\stackrel{\sim}\longrightarrow {\rm qmod}^{(\mathbb{Z}\times \mathbb{Z}/{p\mathbb{Z}})}\mbox{-}\bar{R}.$$

 Consider the following composite equivalence
\begin{align}\label{equ:comp}
\theta^*\circ {\rm res}\colon  {\rm qmod}^L\mbox{-}S \stackrel{\sim}\longrightarrow {\rm qmod}^{(\mathbb{Z}\times \mathbb{Z}/{p\mathbb{Z}})}\mbox{-}\bar{R}.\end{align}
We observe that this equivalence transfers the degree-shift action of  $\mathbb{Z}\vec{\omega}$ on ${\rm qmod}^L\mbox{-}S$ to the degree-shift action of   $\mathbb{Z}/{p\mathbb{Z}}$ on ${\rm qmod}^{(\mathbb{Z}\times \mathbb{Z}/{p\mathbb{Z}})}\mbox{-}\bar{R}$. Consequently, this equivalence induces an equivalence of categories
\begin{align}\label{equ:ZZ}
({\rm qmod}^L\mbox{-}S)^{\mathbb{Z}\vec{\omega}} \stackrel{\sim}\longrightarrow ({\rm qmod}^{(\mathbb{Z}\times \mathbb{Z}/{p\mathbb{Z}})}\mbox{-}\bar{R})^{\mathbb{Z}/{p\mathbb{Z}}}.
\end{align}

By Proposition \ref{prop:app2} we have an equivalence of categories
$$({\rm mod}^{(\mathbb{Z}\times \mathbb{Z}/{p\mathbb{Z}})}\mbox{-}\bar{R})^{\mathbb{Z}/{p\mathbb{Z}}} \stackrel{\sim}\longrightarrow {\rm mod}^\mathbb{Z} \mbox{-}R,$$
which restricts to an equivalence
$({\rm mod}_0^{(\mathbb{Z}\times \mathbb{Z}/{p\mathbb{Z}})}\mbox{-}\bar{R})^{\mathbb{Z}/{p\mathbb{Z}}} \stackrel{\sim}\longrightarrow {\rm mod}_0^\mathbb{Z} \mbox{-}R$
of full subcategories consisting of finite dimensional modules. Applying Corollary \ref{cor:equiv}, we have an equivalence of categories
$$({\rm qmod}^{(\mathbb{Z}\times \mathbb{Z}/{p\mathbb{Z}})}\mbox{-}\bar{R})^{\mathbb{Z}/{p\mathbb{Z}}}\stackrel{\sim}\longrightarrow {\rm qmod}^\mathbb{Z} \mbox{-}R.$$
We combine this equivalence with (\ref{equ:ZZ}) to obtain (1).

For (2), we assume that $k$ is a splitting field of $C_p$. Recall that $C_p$ is identified with the character group of $\mathbb{Z}/{p\mathbb{Z}}$. By Proposition \ref{prop:recover} we have an isomorphism of categories $${\rm mod}^{(\mathbb{Z}\times \mathbb{Z}/{p\mathbb{Z}})}\mbox{-}\bar{R}\stackrel{\sim}\longrightarrow ({\rm mod}^\mathbb{Z}\mbox{-}R)^{C_p},$$
which restricts to an isomorphism  ${\rm mod}_0^{(\mathbb{Z}\times \mathbb{Z}/{p\mathbb{Z}})}\mbox{-}\bar{R}\stackrel{\sim}\longrightarrow ({\rm mod}_0^\mathbb{Z}\mbox{-}R)^{C_p}$  of full subcategories consisting of finite dimensional modules. Applying Corollary \ref{cor:equiv}, we have an equivalence of categories
$${\rm qmod}^{(\mathbb{Z}\times \mathbb{Z}/{p\mathbb{Z}})}\mbox{-}\bar{R}\stackrel{\sim}\longrightarrow ({\rm qmod}^\mathbb{Z}\mbox{-}R)^{C_p}.$$
We combine this equivalence with (\ref{equ:comp}) to obtain (2).
\end{proof}

\begin{rem}\label{rem:main}
(1) We recall that $\mathbb{X}$ is the weighted projective line defined by $S$ and that $\mathbb{E}$ is the projective plane curve defined by $R$. If $k$ is a splitting field of $C_p$, then $\mathbb{E}$ is an elliptic plane curve. Then Theorem \ref{thm:main} relates the categories of coherent sheaves on $\mathbb{X}$ and $\mathbb{E}$ via two equivariantizations.

We emphasize that the two group actions on relevant categories are completely different: one is the degree-shift action and another is the twisting action. However, as in Remark \ref{rem:Len1} the two equivalences obtained in Theorem \ref{thm:main} are somehow \emph{adjoint} to each other. We mention that by Remark \ref{rem:j} the assignment $\mathbb{X}\mapsto\mathbb{E}$ is not a bijection up to isomorphism; see also \cite{Lentalk}.

(2) For the $L$-graded algebra $S$ and the associated $\mathbb{Z}$-graded $R$ one might consider the quotient categories ``${\rm QMod}$" of arbitrary graded modules, which are equivalent to the categories of quasi-coherent sheaves on $\mathbb{X}$ and $\mathbb{E}$, respectively. Then a similar result as Theorem \ref{thm:main} holds for theses two quotient categories ``${\rm QMod}$".

(3) Let us mention the version of Theorem \ref{thm:main} in \cite{Lentalk}.

 Consider the degree map $\delta\colon L\rightarrow \mathbb{Z}$ whose kernel equals the torsion subgroup ${\rm tor}(L)$ of $L$; moreover, we have a decomposition $L={\rm tor}(L)\oplus \mathbb{Z}\vec{x}_1$. Consider the $\mathbb{Z}$-graded algebra $\underline{S}=\delta_*(S)$; see Subsection 5.1. We view the $L$-graded algebra $S$ as a ${\rm tor}(L)$-refinement of $\underline{S}$. Observe that $\mathbb{Z}\vec{\omega}\subseteq {\rm tor}(L)$.

 If the field $k$ is a splitting field of $C_p$, then it is also a splitting field of ${\rm tor}(L)$. In this case, we denote by $A$ the character group of ${\rm tor}(L)$.  Then $A$ acts on $\underline{S}$ as $\mathbb{Z}$-graded algebra automorphisms by Lemma \ref{lem:corres}. We consider the \emph{degree-shift action} of ${\rm tor}(L)$ on ${\rm qmod}^L\mbox{-}S$, and the \emph{twisting action} of $A$ on ${\rm qmod}^\mathbb{Z}\mbox{-}\underline{S}$. Then under the same assumptions as Theorem \ref{thm:main}, the same argument yields the following two adjoint equivalences:
  \begin{align}\label{equ:LM}
  ({\rm qmod}^L\mbox{-}S)^{{\rm tor}(L)} \stackrel{\sim}\longrightarrow {\rm qmod}^\mathbb{Z}\mbox{-}\underline{S}, \quad {\rm qmod}^L\mbox{-}S\stackrel{\sim}\longrightarrow ({\rm qmod}^\mathbb{Z}\mbox{-}\underline{S})^{A}.\end{align}
Indeed, the argument here is easier since it does not rely on Proposition \ref{prop:eff} and Theorem \ref{thm:R-S}.

The major difference between the  two equivalences (\ref{equ:LM}) and the ones in Theorem \ref{thm:main} is the fact that  the acting groups in (\ref{equ:LM}) are not cyclic except the case $\mathbf{p}=(6,3,2)$, while all the  acting groups in Theorem \ref{thm:main} are cyclic.

We mention that for the cases $\mathbf{p}=(4,4,2)$ and $(6,3,2)$, the variables of the $\mathbb{Z}$-graded algebra $\underline{S}$ are not all of degree one. Therefore, to verify that $ {\rm qmod}^\mathbb{Z}\mbox{-}\underline{S}$ is equivalent to the category of coherent sheaves on an elliptic curve, one might need the general fact on Veronese subalgebras \cite[Proposition 5.10]{AZ}, or the axiomatic description of the category of coherent sheaves on an elliptic curve in \cite[Theorem C]{RV}. If the characteristic of $k$ is not equal to $2$, the $\mathbb{Z}$-graded algebra $\underline{S}$ for the case $\mathbf{p}=(2,2,2,2)$ defines an elliptic space curve in its Jacobi form. If the characteristic of $k$ is not equal to $3$, the $\mathbb{Z}$-graded algebra $\underline{S}$ for the case $\mathbf{p}=(3,3,3)$ defines an elliptic curve as a Fermat curve of degree $3$.
\end{rem}

\vskip 15pt

\noindent {\bf Acknowledgements}\; J. Chen is very grateful to Helmut Lenzing for giving her the manuscript \cite{Len} during his visit at Xiamen University in 2011.  The authors thank Helmut Lenzing very much for enlightening and helpful discussions. The authors are grateful to Henning Krause and Hagen Meltzer for helpful comments.

J. Chen and Z. Zhou are supported by National Natural Science Foundation of China (No. 11201386) and Natural Science Foundation of Fujian Province (No. 2012J05009); X.W. Chen is supported by NCET-12-0507, National Natural Science Foundation of China (No. 11201446) and the Alexander von Humboldt Stiftung.

\bibliography{}

\vskip 15pt

 \noindent {\tiny

 \vskip 5pt

\noindent Jianmin Chen, and Zhenqiang Zhou\\
School of Mathematical Sciences, \\
Xiamen University, Xiamen, 361005, Fujian, PR China.\\
E-mail: chenjianmin@xmu.edu.cn, zhouzhenqiang2005@163.com.\\}
\vskip 3pt

{\tiny \noindent   Xiao-Wu Chen \\
 School of Mathematical Sciences,\\
  University of Science and Technology of China, Hefei 230026, Anhui, PR China \\
  Wu Wen-Tsun Key Laboratory of Mathematics,\\
 USTC, Chinese Academy of Sciences, Hefei 230026, Anhui, PR China.\\
E-mail: xwchen@mail.ustc.edu.cn, URL: http://home.ustc.edu.cn/$^\sim$xwchen.}

\end{document}